\documentclass[12pt,reqno]{amsart}

\usepackage{latexsym}
\usepackage{amssymb, tikz}

\usepackage{color}

\usepackage{amsthm}
\theoremstyle{plain}

\usepackage{calrsfs}
\DeclareMathAlphabet{\pazocal}{OMS}{zplm}{m}{n}

\newtheorem*{theorem*}{Theorem}

\renewcommand{\epsilon}{\varepsilon}

\newcommand{\N}{{\mathbb N}}
\newcommand{\R}{{\mathbb R}}
\newcommand{\C}{{\mathbb C}}

\newcommand{\Z}{{\mathbb Z}}

\newcommand{\E}{{\mathbf E}}

\newcommand{\vol}{{\operatorname{Vol}}}

\newcommand{\Oc}{\mathcal{O}}

\numberwithin{equation}{section}

\newtheorem{theorem}{Theorem}[section]

\newtheorem{theo}[theorem]{{\sc Theorem}}

\newtheorem{cor}[theorem]{{\sc Corollary}}

\newtheorem{lem}[theorem]{{\sc Lemma}}

\newtheorem{prop}[theorem]{{\sc Proposition}}
\newtheorem{lemma}[theorem]{{\sc Lemma}}
\newtheorem{definition}[theorem]{{\sc Definition}}

\theoremstyle{definition}
\newtheorem{defn}[theorem]{{\sc Definition}}

\theoremstyle{remark}
\newtheorem{rem}[theorem]{Remark}

\newcommand{\prob}{\pazocal{P}r}

\newcommand{\Dc}{\mathcal{D}}

\newcommand{\Sc}{S}

\newcommand{\Ncc}{\mathcal{N}}

\newcommand{\gfr}{\mathfrak{g}}

\newcommand{\nod}{\pazocal{N}}

\newcommand {\M} {\pazocal{M}}

\newcommand {\Lc} {\mathcal{L}}

\newcommand {\Tb} {\mathbb{T}}

\title[Direction distribution]{Direction distribution for nodal components of random band-limited functions on surfaces}
\author{Suresh Eswarathasan}
\address{Department of Mathematics, Dalhousie University, Chase Building, Halifax, Nova Scotia, Canada}
\email{sr766936@dal.ca}
\author{Igor Wigman}
\address{Department of Mathematics, King's College London, Strand Campus, London, United Kingdom}
\email{igor.wigman@kcl.ac.uk}

\date{\today}

\begin{document}

\begin{abstract}
Let $(\M,g)$ be a smooth compact Riemannian surface with no boundary. Given a smooth vector field $V$ with finitely many zeroes on $\M$, we study the distribution of the number of tangencies to $V$ of the nodal components of random band-limited functions. It is determined that in the high-energy limit, these obey a universal deterministic law, independent of the surface $\M$ and the vector field $V$, that is supported precisely on the even integers $2 \Z_{> 0}$.
\end{abstract}

\maketitle


\section{Introduction}

\subsection{Nodal components of random functions}

Given a ``nice" function
$f:\R^{d}\rightarrow\R$, $d\ge 2$, or $f:\M\rightarrow\R$ with $\M$ a compact Riemannian $d$-manifold, the nodal set
of $f$ is its zero set $f^{-1}(0)$; if $f$ is Morse, then its nodal set is a smooth hypersurface in both the ``Euclidean" and
``Riemannian" contexts.
In either scenario, the {\em nodal components} of $f$ are the connected components of the nodal set, and
the {\em nodal domains} are the (positive or negative) connected components of the {\em complement} of the nodal set.
The most basic question one is interested is in the {\em nodal count} of $f$, i.e.
the total number of the nodal components of $f$, also yielding the total number of nodal domains of $f$ via Euler's identity,
at least, in the Riemannian context $f:\M\rightarrow\R$ (where the nodal set is of finite
hypersurface volume, thanks to the assumed compactness of $\M$).

To make sense of the analogous question in the Euclidean context $f:\R^{d}\rightarrow\R$, one usually takes a large parameter $R>0$
and studies the asymptotics of the nodal count of the restriction $f|_{B(R)}$ to the centred radius-$R$ ball $B(R)\subseteq \R^{d}$, as
$R\rightarrow\infty$. Other than the nodal count, one might refine the said question by separately counting the nodal components of $f$ belonging to
e.g. a given diffeomorphism type or otherwise, study the mutual positions of the components (``nestings") etc.

\vspace{2mm}

Understanding the ``typical" nature of the nodal structures of {\em Gaussian random fields}, rather than individual functions,
is an actively pursued subject within several disciplines, in the last few years in particular.
Let $F:\R^{2}\rightarrow \R$ (more generally, $F:\R^{d}\rightarrow\R$) be a Gaussian random field, that will be assumed {\em stationary},
and $R>0$ be a large parameter.
Then the number $\nod(F,R)$ of nodal components of $F$ {\em fully contained} inside $B(R)$ is a random variable; alternatively, one
could count those components merely {\em intersecting} $B(R)$. It was shown by Nazarov-Sodin ~\cite{NS09,So12,NS15} that, under very
mild smoothness and non-degeneracy assumptions on the law of $F$, there exists a constant $c_{NS}\ge 0$ so that
$\frac{\nod(F,R)}{\vol(B(R))}$ converges to $c_{NS}$ both a.s. and in mean, i.e.
\begin{equation}
\label{eq:cNS conv mean}
\E\left[\left| \frac{\nod(F,R)}{\vol(B(R))} - c_{NS}  \right|\right] \rightarrow 0
\end{equation}
as $R\rightarrow\infty$. Sarnak-Wigman ~\cite{SW18}, and Beliaev-Wigman ~\cite{BW17} further developed the techniques
due to Nazarov-Sodin by considering the more refined questions of separate nodal counts belonging to particular diffeomorphism types or given hypersurface measure.

\subsection{Random band-limited functions on smooth manifolds} \label{s:band_lim}

Rather than for its own intrinsic interest, this Euclidean scenario above serves as a ground state for the Riemannian one, namely,
as its scaling limit. Here we consider
a sequence (``ensemble") $\{f_{l}\}_{l\in\Lc}$, of smooth random Gaussian fields $f_{l}:\M\rightarrow\R$, satisfying a natural scaling
property, with the scaling parameter $l$ lying in some countable set $\Lc$, and our objective is to study the distribution
of the {\em total} number of nodal components of $f_{l}$ as $l\rightarrow\infty$, their typical topology, geometry, relative positions,
and other important properties. A particularly important such ensemble,
motivating the work ~\cite{SW18}, is
the ensemble of band-limited functions, depending on a fixed number $\alpha\in [0,1]$. This includes the important ensembles of
random degree-$l$ spherical harmonics (see \S\ref{sec:spher harm RMW} below), and Arithmetic Random Waves (\S\ref{sec:ARW} below).

It is well-known that, since we assumed $\M$ to be a smooth, compact, Riemannian $d$-manifold, the space $L^{2}(\M)$ of square-summable functions on $\M$
has an orthonormal basis $\{\varphi_{j}\}_{j=1}^{\infty}$ consisting of Laplace eigenfunctions, i.e.
\begin{equation}
\label{eq:Helmholts}
\Delta \phi_{j}+t_{j}^{2}\phi_{j} = 0,
\end{equation}
where $\Delta$ is the Laplace-Beltrami operator on $\M$ acting on $L^{2}$, and $\{t_{j}\}_{j\ge 0}$ is its purely discrete spectrum
$$0=t_{0}\le t_{2}\le\ldots,$$ satisfying $t_{j}\rightarrow\infty$. For a ``band" $\alpha\in [0,1)$ and spectral
parameter $T>0$ (with the intention of taking the limit $T\rightarrow\infty$), we define ~\cite{SW18} the random
band-limited functions to be
\begin{equation}
\label{eq:fT band lim alpha<1}
f_{T}(x)=f_{\alpha;T}(x)=\frac{1}{\left|\{j:\alpha\cdot T<t_{j}<T\}\right|^{1/2}}\sum\limits_{\alpha\cdot T<t_{j}<T}c_{j}\varphi_{j}(x),
\end{equation}
where the $c_{j}$ are i.i.d. standard Gaussian. For $\alpha=1$ the above definition usually makes no sense, as the summation on the r.h.s.
of \eqref{eq:fT band lim alpha<1} is typically one summand (or empty), so in this case we interpret the summation as
\begin{equation}
\label{eq:fT band lim alpha=1}
f_{1;T}(x)=\frac{1}{\left|\{j:T-\eta(T)<t_{j}<T\}\right|^{1/2}}\sum\limits_{T-\eta(T)<t_{j}<T}c_{j}\varphi_{j}(x),
\end{equation}
with the convention $\eta(T)=o_{T\rightarrow\infty}(T)$ but $\eta(T)\rightarrow\infty$.

Alternatively, $f_{T}(\cdot)$ is the (uniquely defined) centred Gaussian random field of covariance
\begin{equation}
\label{eq:KT covar func}
K_{T}(x,y)=K_{\alpha;T}:=\E[f_{\alpha;T}(x)\cdot f_{\alpha;T}(y)] = \sum \phi_{j}(x)\cdot\phi_{j}(y),
\end{equation}
$x,y\in\M$, where the summation is over the same energy window as \eqref{eq:fT band lim alpha<1} for $\alpha<1$
(resp. \eqref{eq:fT band lim alpha=1} for $\alpha=1$); we identify $K_{T}(\cdot,\cdot)$ as the {\em spectral projector}
for the corresponding energy window. The kernel $K_{T}(\cdot, \cdot)$ and its derivatives possess scaling limits once appropriately scaled by $T$,
see \S\ref{sec:proofs outline} for more details, and in particular \eqref{eq:covar func band lim der conv}.

\subsubsection{Random spherical harmonics and Berry's Random Wave Model}
\label{sec:spher harm RMW}

It is well-known that the Laplace eigenfunctions on the $2$-sphere $\Sc^{2}\subseteq \R^{3}$, that is the spherical harmonics, are restrictions of harmonic polynomials of some degree $l\ge 1$. The space of degree-$l$ spherical harmonics is of
dimension $2l+1$, so given a number $l$ we may obtain an $L^{2}$-orthonormal basis $\Phi_{l}:=\{\eta_{l;1},\ldots,\eta_{l;2l+1}\}$,
and define
\begin{equation}
\label{eq:Tl spher harm}
T_{l}(x):=\frac{\sqrt{4\pi}}{\sqrt{2l+1}}\sum\limits_{m=1}^{2l+1}a_{m}\eta_{l;m}(x)
\end{equation}
with $\{a_{m}\}_{m=1}^{2l+1}$ standard Gaussian i.i.d.; the law of the random spherical harmonics $T_{l}(\cdot)$ is
invariant w.r.t. the choice of $\Phi_{l}$. The random fields $T_{l}(\cdot)$ are the Fourier components of every
rotation invariant random field on $\Sc^{2}$, hence its importance in a variety of disciplines within mathematics,
physics and cosmology.

Equivalently, the random spherical harmonic $T_{l}(\cdot)$ is a centred
Gaussian random field defined via the covariance function $$r_{l}(x,y)=P_{l}(\cos d(x,y)),$$ $x,y\in\Sc^{2}$,
$d(\cdot,\cdot)$ is the spherical geodesic distance, and $P_{l}(\cdot)$ is the degree-$l$ Legendre polynomial.
Since, by the standard Hilb's asymptotics ~\cite{Szego}, $$P_{l}(\cos\theta)\approx J_{0}((l+1)\theta),$$
the geometry of the nodal line of $T_{l}$ could be compared to the geometry of the nodal line of the stationary isotropic
random field on $\R^{2}$ defined by the covariance function $J_{0}(\|x\|)$, usually referred to ``Berry's Random Wave Model"
(RWM), for it is postulated ~\cite{Berry 1977} to be a stand-in for deterministic Laplace eigenfunctions on generic chaotic surfaces.

\subsubsection{Arithmetic random waves}
\label{sec:ARW}

The Arithmetic Random Waves are random Gaussian toral Laplace eigenfunctions $\{f_{n}:\Tb^{2}\rightarrow\R\}_{n\in S}$,
where $\Tb^{2}=\R^{2}/\Z^{2}$ is the $2$-dimensional standard torus, and $S=\{a^{2}+b^{2}:\:a,b\in\Z\}$ is the set of all numbers
expressible as sum of two integer squares. For $n\in S$ let $$\Lambda_{n}:=\{(a,b)\in \Z^{2}:\: a^{2}+b^{2}=n\}$$ be the set
of all lattice points lying on the centred circle of radius $\sqrt{n}$. We may define
\begin{equation}
\label{eq:fn ARW def}
f_{n}(x)=\sum\limits_{\lambda\in\Lambda_{n}} a_{\lambda}e(\langle \lambda,x  \rangle),
\end{equation}
where $e(t)=e^{2\pi i t}$, and the $\{a_{\lambda}\}_{\lambda\in\Lambda_{n}}$ are complex standard Gaussian i.i.d., save to the condition
$a_{-\lambda}=\overline{a_{\lambda}}$, so that $f_{n}$ are real valued.

The random fields $f_{n}$ are the ``Arithmetic Random Waves"
that serve as a motivation to our research, following the work ~\cite{RudWig2018}; $f_{n}$ should be compared
to the band-limited functions \eqref{eq:fT band lim alpha=1}, with $n\approx T^{2}$
(and, thanks to the spectral multiplicity, we may take an infinitesimal energy window).
Equivalently to the explicit definition \eqref{eq:fn ARW def},
$f_{n}$ could be defined as the stationary centred Gaussian random field on $\Tb^{2}$ with the covariance function
\begin{equation}
\label{eq:ARW covar def}
r_{n}(x)=\E[f_{n}(y)\cdot f_{n}(y+x)] = \frac{1}{|\Lambda_{n}|}\sum\limits_{\lambda\in\Lambda_{n}}
e(\langle \lambda,x \rangle),
\end{equation}
or via the spectral measure of $f_{n}$, i.e. the atomic measure
\begin{equation}
\label{eq:ARW spec measure rho}
\rho_{n}:=\frac{1}{|\Lambda_{n}|}\sum\limits_{\lambda\in\Lambda_{n}}\delta_{\lambda/\sqrt{n}}
\end{equation}
on $\Sc^{1}\subseteq\R^{2}$.

\subsection{Direction distribution and statement of the principal result}

\subsubsection{Direction distribution on $\mathbb{T}^d$}

In $2$d Rudnick and Wigman ~\cite{RudWig2018} proposed to study the {\em ``direction distribution"} of the nodal line, a different quantity they
introduced for functions defined on the $2$d standard torus $\Tb^{2}=\R^{2}/\Z^{2}$, also related to ~\cite{Swerling}.
Let $\zeta\in\Sc^{1}$ be
a direction, and $f:\Tb^{2}\rightarrow\R$ any ``nice" function. The number
\begin{equation}
\label{eq:Nzeta torus}
\mathcal{N}_{\zeta}(f):= \left|\left\{x\in \Tb^{2}:f(x)=0,\, \frac{\nabla f(x)}{\|\nabla f(x)\|}=\pm \zeta  \right\} \right|
\end{equation}
is the number of
points on the nodal line of $f$ normal to $\zeta$; equivalently, $\mathcal{N}_{\zeta}(f)$ is the number of nodal points of $f$ in direction
$\pm\xi:=\zeta^{\perp}$. The direction distribution carries a lot of information on the nodal line of $f$; for example, the number of
nodal components of $f$ on $\Tb^{2}$ essentially majorizes ~\cite{KuWiAdv} the nodal count of $f$ for every $\zeta$,
as every nodal component of $f$ of trivial
homology contains at least two points tangent to $\xi$, and it is usually easy to control the contribution of all the other
components.

For the $\mathcal{N}_{\zeta}(\cdot)$ corresponding to
toral Laplace eigenfunctions Rudnick and Wigman ~\cite{RudWig2018} gave the optimal upper bounds, as well as evaluated their total expected number for the associated Gaussian random model, ``Arithmetic Random Waves" \eqref{eq:fn ARW def}, precisely, appealing to the Kac-Rice method. Their result, however, does not allow for the separation of the $\xi$-tangencies into nodal components of $f$, i.e. how many points on a
{\em given component} represent $\xi$, or how many nodal components of $f$ represent the direction $\xi$ {\em precisely} $k$ times.

\subsubsection{Direction distribution on curved surfaces}

The purpose of this manuscript is two-fold. First, we propose, what it seems, a natural generalisation of the direction distribution concept for smooth manifolds with curvature, and study its properties for the random band-limited functions \eqref{eq:fT band lim alpha<1} (and
\eqref{eq:fT band lim alpha=1}).
Second, in this random setting (random band-limited functions, including, in particular Arithmetic Random Waves) we aim at refining
the said results due to Rudnick-Wigman by keeping separate accounts for the number of $\xi$-tangencies on individual components rather than merely the total number of $\xi$-tangencies (though the latter is useful for our purposes).

Let us now introduce a space of vector fields that play a central role in our article. We let $\mathcal{V}(\M)$ be the class of all $C^{\infty}$-smooth vector fields on $\M$, with finitely many zeros; the class $\mathcal{V}(\M)$ is non-empty for
every smooth $\M$ by Lemma \ref{l:vfs_finite_zeroes}. First, given $f:\M\rightarrow\R$ a nice function, and $V\in\mathcal{V}(\M)$, we introduce
\begin{equation*}
\mathcal{N}_{V}(f)= \#\{x:\:V(x)\ne 0,\, f(x)=Vf(x)=0\},
\end{equation*}
the announced generalisation of the ``flat" direction distribution, with the variable direction $V(x)$ in place of
$\xi=\zeta^{\perp}$ in \eqref{eq:Nzeta torus}. Other than for the torus, a precise asymptotic expression for the total number of
$V$-tangencies for the random spherical harmonics \eqref{eq:Tl spher harm} was recently obtained \cite{E18} by the first author.  Related results were obtained by Dang and Rivi\`ere \cite{DaRiJEMS} following some substantial and general results of Gayet and Welschinger \cite{GaWeJussieu}.

Now we separate the tangency counts for individual nodal components of $f$: for a nodal component $\gamma\subseteq f^{-1}(0)$ define
the number of tangencies w.r.t. $V$,
\begin{equation*}
T_{V}(\gamma)= \#\{x\in \gamma:\: V(x)\ne 0,\, f(x)=Vf(x)=0\},
\end{equation*}
and
\begin{equation*}
\nod_{V}(f,k) :=\#\{\gamma  \subseteq f^{-1}(0):\: T_{V}(\gamma)=k\}
\end{equation*}
is the total number of nodal components of $f$ with precisely $k$ tangencies w.r.t. $V$,
so that we have the
identity
\begin{equation}
\label{eq:tot numb tang=sum mean}
\mathcal{N}_{V}(f) = \sum\limits_{k=1}^{\infty}k\cdot \nod_{V}(f,k).
\end{equation}
Finally, we let
\begin{equation*}
\nod (f) = \sum\limits_{k=0}^{\infty}\nod_{V}(f,k)
\end{equation*}
be the total number of the nodal components of $f$.

\subsubsection{Statement of the principal result}

Our main result concerns the asymptotic law for the direction distribution measure corresponding to the band-limited functions
$f_{\alpha;T}(\cdot)$ in \eqref{eq:fT band lim alpha<1}-\eqref{eq:fT band lim alpha=1}.
Following the approach of Sarnak-Wigman ~\cite{SW18},
we may incorporate, or ``encapsulate", all the individual counts $\nod_{V}(f,\cdot)$ into a single (random) probability measure,
the ``direction distribution measure",
\begin{equation}
\label{eq:mu(f) ddm def}
\mu_{f}(V) =\frac{1}{\nod(f)}\sum\limits_{\gamma\subseteq f^{-1}(0)}\delta_{T_{V}(\gamma)} =\frac{1}{\nod(f)}\sum\limits_{k=0}^{\infty}\nod_{V}(f,k)\cdot\delta_{k},
\end{equation}
on $\Z_{\ge 0}$.
Given two probability measures $\mu_{1},\mu_{2}$ on $\Z$ we will use the total variation distance function
\begin{equation}
\label{eq:D tot var dist def}
\Dc(\mu_{1},\mu_{2}) = \sup\limits_{F\subseteq \Z_{\ge 0}} |\mu_{1}(F)-\mu_{2}(F)|.
\end{equation}

\begin{theo}
\label{thm:main meas conv}

Given $\alpha\in [0,1]$, there exists a (deterministic) probability measure $\mu_{\alpha}$ on $\Z_{\ge 0}$,
supported on the positive even integers $2\Z_{>0}$, so that for all $V\in\mathcal{V} (\M)$ and every $\epsilon>0$,
\begin{equation}
\label{eq:main meas conv}
\lim\limits_{T\rightarrow\infty}\prob\left(\Dc(\mu_{f_{\alpha;T}}(V),\mu_{\alpha})>\epsilon\right)=0,
\end{equation}
where $\Dc(\cdot,\cdot)$ is the total variation distance \eqref{eq:D tot var dist def}.

\end{theo}

\subsubsection{On the principal result}

The band-limited functions \eqref{eq:fT band lim alpha<1}-\eqref{eq:fT band lim alpha=1} is a particular case of an aforementioned ensemble of Gaussian random fields possessing a natural scaling by the wave number $T$, see \S\ref{sec:proofs discussions} below, and, in particular \eqref{eq:covar func band lim der conv}.
Our proofs are applicable for such ensembles with scaling,
in a more general scenario than merely the band-limited functions, as long as an analogue of \eqref{eq:covar func band lim der conv} holds,
with $T$ replaced by the scaling parameter (see e.g. the important application on Arithmetic Random Waves from \S\ref{sec:ARW},
in \S\ref{sec:appl ARW} below).
In general, in a scenario like \eqref{eq:main meas conv}, when probability measures (on $\Z_{\ge 0}$) weak-$*$
converge to a limit measure $\mu_{\alpha}$, it forces the limit measure to have $\mu_{\alpha}(\Z_{\ge 0}) \le 1$
by Fatou's Lemma, and $\mu_{\alpha}(\Z_{\ge 0}) < 1$ would mean escape of probability to infinity. As part
of Theorem \ref{thm:main meas conv}, we will be able to rule this kind of losses of mass, cf. \cite[Theorem 1.1]{SW18}.

We stress that, despite the fact that the measures $\mu_{f_{\alpha;T}}$ are {\em random}, the limit measure $\mu_{\alpha}$
is {\em deterministic}, and, notably, in addition, $\mu_{\alpha}$ is {\em independent} of $V$.
Thanks to the asymptotic scaling \eqref{eq:covar func band lim der conv}, we will be able to compare the nodal counts $\nod_{V}(f,k)$ to
the nodal counts of some Gaussian isotropic random field $\gfr_{\alpha}(\cdot)$ on $\R^{2}$
(to be defined in \S\ref{sec:proofs outline}),
with $V\equiv \xi$ a constant vector field,
up to an admissible error term. The measures $\mu_{\alpha}$ are $V$-independent, since, by the isotropic property of $\gfr_{\alpha}(\cdot)$,
the corresponding distribution of $\xi$-tangencies counts is independent of $\xi$.
The support of $\mu_{\alpha}$ is a manifestation of the fact that the number of $\xi$-tangencies
for a simple smooth planar curve is necessarily {\em even}, unless a degeneracy occurs, with probability $0$ (though,
strictly speaking, it is possible to construct simple curve with an odd number of $\xi$-tangencies). However,
for $k>0$ even it is possible to construct a simple curve with number of $\xi$-tangencies being precisely $k$, that occurs
as a nodal component of the scaled random field with positive density (the most subtle or delicate case being $\alpha=1$, cf.
~\cite[Proposition $5.3$]{SW18} and ~\cite{CanSar18}).
We refer the reader to \S\ref{sec:proofs outline} for more details on proofs and intuitions as of why these peculiarities hold.

\subsubsection{$T^*\M$ and conormal cycles}
We conclude this section by emphasizing that the microlocal geometry of nodal sets, or more specifically the interaction between ``cotangent/conormal spaces" to nodal sets and various other geometric quantities of these submanifolds, is a natural topic of study. The work of Dang and Rivi\`ere gives asymptotics pertaining to the equidistribution (in $T^* \M$) of ``conormal cycles" for $f_T$ on general compact manifolds \cite{DaRiJEMS}, with conormal cycles being phase-space quantities related to the conormal bundle of $f^{-1}_T(0)$, namely $N^*(\{ f_T = 0 \}) \subset T^* \M$.  In other words, for example in odd dimensions, the authors show that the expected value of a natural and re-scaled volume measure on the conormal bundle to the nodal set $N^*(\{ f_T = 0 \})$ converges to a uniform volume measure on $T^*\M$ as $T \rightarrow \infty$.
It is worth noting that Dang-Rivi\`ere were themselves motivated by the work \cite{GaWeJussieu} on expected Betti numbers for nodal sets associated to elliptic pseudodifferential operators; a local refinement of their lower bound was very recently obtained by Wigman \cite{Wig19}.

\subsection{Applications to arithmetic random waves}

\label{sec:appl ARW}

Recall the Arithmetic Random Waves $f_{n}$ in \eqref{eq:fn ARW def} for $n\in S$, the set of numbers expressible as sum of two squares, the corresponding covariance function \eqref{eq:ARW covar def},
and their spectral measure \eqref{eq:ARW spec measure rho}.
It is well-known ~\cite{KaKo,ErHa},
that for a ``generic" (density-$1$) sequence $\{n\}\subseteq S$, the measures $\rho_{n}$ equidistribute
on $\Sc^{1}$, i.e.
\begin{equation}
\label{eq:rhon equidistr}
\rho_{n}\Rightarrow \frac{d\theta}{2\pi}.
\end{equation}
However, there exist ~\cite{Cil,KuWiMAnn} other weak-$*$ partial limits of the sequence $\{\rho_{n}\}_{n\in S}$.
It was established ~\cite{KKW,So12,KuWiAdv}, that for the nodal structures of $f_{n}$ to exhibit a limit law,
it is essential to divide $S$ into
subsequences whose corresponding $\rho_{n}$ obey a limiting distribution, i.e. take a subsequence $\{n\}\subseteq S$
so that $\rho_{n}\Rightarrow\tau$ for some probability measure $\tau$ on $\Sc^{1}$. In this case
the corresponding covariance functions $r_{n}$ converge uniformly locally to $r$, the (inverse) Fourier transform of
$\tau$, in the sense of \eqref{eq:covar func band lim der conv}.

Let $\nod(f_{n})$ be the total number of nodal components of $f_{n}$. A combination of the techniques by Nazarov-Sodin
~\cite{So12,NS15} and ~\cite{KuWiAdv} implies that for every sequence $\{n\}\subseteq S$ so that
\begin{equation*}
\rho_{n}\Rightarrow \tau
\end{equation*}
for some probability measure $\tau$ on $\Sc^{1}$, there exists a constant $c_{NS}(\tau)\ge 0$ so that
\begin{equation*}
\E[\nod(f_{n})]= c_{NS}(\tau)\cdot n +o_{n\rightarrow\infty}(1).
\end{equation*}
Moreover, $c_{NS}(\tau)=0$, if and only if $\tau$ is either the ``Cilleruelo" measure
\begin{equation*}
\tau_{0}=\frac{1}{4}\left(\delta_{\pm 1}+\delta_{\pm i}\right),
\end{equation*}
i.e. the atomic probability measure supported on the $4$ antipodal points $\pm 1$ and $\pm i$ (viewing $\R^{2}\cong \C$),
or its tilted by $\pi/4$ variant
\begin{equation*}
\widetilde{\tau_{0}}=\frac{1}{4}\left(\delta_{\pm \pi/4}+\delta_{3\pi/4}\right).
\end{equation*}
The probability measures $\tau_{0}$ and $\widetilde{\tau_{0}}$ on $\Sc^{1}$ are the only measures
satisfying all the underlying symmetries (invariance w.r.t. rotation by $\pi/2$ and complex conjugation) that
are supported on $4$ points only.

\vspace{2mm}

Though not applicable as a black box, our Theorem \ref{thm:main meas conv} (or, rather, the associated techniques)
yields the following extension of the said results
concerning the $V$-direction distribution measures $\mu_{f_{n}}(V)$ of $f_{n}$
for some $V\in\mathcal{V}(\Tb^{2})$ (for example, $V$ could be the constant vector field
$V\equiv \xi$), defined as in \eqref{eq:mu(f) ddm def}.
Let $V\in\mathcal{V}(\Tb^{2})$, and $\{n\}\subseteq S$ be a subsequence so that
\begin{equation*}
\rho_{n}\Rightarrow \tau
\end{equation*}
for some $\tau\ne \tau_{0},\widetilde{\tau_{0}}$.
Then there exists a deterministic probability measure $\mu_{\tau;V}$ on $\Z_{\ge 0}$, so that
\begin{equation*}
\Dc\left( \mu_{f_{n}}(V), \mu_{\tau;V}  \right) \rightarrow 0,
\end{equation*}
as $n\rightarrow\infty$. In addition, the support of $\mu_{\tau;V}$ is contained in the set of positive even
integers $2\Z_{>0}$.

We emphasize that for the Arithmetic Random Waves, unlike the band-limited functions in the statement of Theorem \ref{thm:main meas conv},
the limit direction distribution
measure $\mu_{\tau;V}$ does depend on $V$, and we are not aware whether it is possible to explicate the dependency of $\mu_{\tau;V}$ on
$V$; this is a by-product
of the fact that the spectral measure $\tau$ is not invariant w.r.t. rotations. Moreover, the support might fail to attain the whole of $2\cdot\Z_{>0}$,
depending on $\tau$ (and possibly $V$). However, for the generic case \eqref{eq:rhon equidistr},
the conclusions of Theorem \ref{thm:main meas conv} hold true in their full strength with $\mu_{d\theta/2\pi}=\mu_{1}$ same
as in \eqref{eq:main meas conv}, and, in particular,
$\mu_{d\theta/2\pi}=\mu_{d\theta/2\pi;V}$ does not depend on $V$, and its support equals to precisely $2\Z_{>0}$.

\subsection*{Acknowledgements}

We warmly thank Ze\'{e}v Rudnick for suggesting this collaboration, his interest in our work and his support,
and to Manjunath Krishnapur and Lakshmi Priya for a very enlightening discussion during the conference
``Random Waves in Oxford" that took place in June $2018$. Thanks also to Dmitry Beliaev, Matthew de Courcy-Ireland, Gabriel Rivi\`ere, Nguyen Viet Dang, and Damien Gayet for various discussions on their own work which themselves gave us insights.  A part of the presented research was
conducted during both authors' visit at the Banff International Research Station (BIRS) in July $2019$,
and we are grateful to them for providing us with a unique opportunity to use the outstanding facilities of BIRS as part of the
programme ``Research in Pairs". SE was in residence at McGill University as a Visiting Professor for a majority of the writing of this article and warmly thanks the Department of Mathematics for their hospitality. However, the discussions that lead to the writing of this article began while SE was Lecturer in Analysis at Cardiff University; he is grateful for the opportunity to work in a such stimulating environment. The research leading to these results has received funding from the European Research
Council under the European Union's Seventh Framework Programme (FP7/2007-2013), ERC grant agreement n$^{\text{o}}$ 335141 (I.W.).

\section{Outline of the proofs and discussion}

\label{sec:proofs discussions}

\subsection{Outline of the proofs}

\label{sec:proofs outline}

It is known
that ~\cite{Lax,Horm,CH1,CH2, S88, SV96}, under the assumptions above, the covariance functions $K_{T}(\cdot, \cdot)$ as in \eqref{eq:KT covar func}
scales by $T$ around every point of $\M$, in the following sense.
For $x\in \M$, and $u,v\in \R^{2}$ so that both $\frac{\|u\|}{T},\frac{\|v\|}{T}$ are less than the injectivity
radius of $x$ define the {\em scaled} covariance function using the identification $I_{x}:\R^{2}\rightarrow T_{x}\M$:
\begin{equation*}
K_{T;x}(u,v):= K_{T}(\exp_{x}(I_{x}(u/T)),\exp_{x}(I_{x}(v/T)))
\end{equation*}
on $u,v\in\R^{2}$, corresponding to the {\em scaled} random fields
\begin{equation}
\label{eq:fxT scaled def}
f_{x,T}(u)=f_{\alpha;x,T}(u)=f_{T}\left(\exp_{x}(I_{x}(u/T))\right),
\end{equation}
$u\in\R$ ($\|u\|$ smaller than the injectivity radius of $x$).

Then, for all $u,v\in \R^{2}$ {\em fixed},
\begin{equation}
\label{eq:KTx(u,v)->B(|u-v|)}
K_{T;x}(u,v)\rightarrow B(\|u-v\|),
\end{equation}
locally uniformly, where
\begin{equation}
\label{eq:Bu Fourier annulus}
B(u):= \frac{1}{|A_{\alpha}|}\int\limits_{A_{\alpha}}e^{-2\pi i \langle u,v\rangle} dv
\end{equation}
is the Fourier transform of the characteristic function of the annulus $$A_{\alpha}=\{v\in\R^{2}:\: \alpha<\|v\|<1\}$$
(the unit circle $\Sc^{1}\subseteq\R^{2}$ for $\alpha=1$),
and $|A_{\alpha}|$ is the volume of $A_{\alpha}$ (resp. the length $\pi$ of $\Sc^{1}$ for $\alpha=1$).
Explicitly, \eqref{eq:KTx(u,v)->B(|u-v|)} means that for every $R>0$,
\begin{equation}
\label{eq:covar func band lim der conv}
\sup\limits_{\|u\|,\|v\|<R} \left| K_{T;x}(u,v)- B_{\alpha}(\|u-v\|)\right|\rightarrow 0,
\end{equation}
and the same holds for all derivatives of $K$, where the rate of convergence in \eqref{eq:covar func band lim der conv}
depends on the order of the derivative only.

The rotation invariant function $B_{\alpha}(\cdot)$ defines a stationary isotropic random field
$\gfr_{\alpha}:\R^{2}\rightarrow\R$, i.e. the covariance
function of $\gfr_{\alpha}(\cdot)$ is $$\E[\gfr_{\alpha}(x)\cdot \gfr_{\alpha}(y)]=B_{\alpha}(\|x-y\|).$$
By \eqref{eq:Bu Fourier annulus}, the spectral measure of $\gfr_{\alpha}$ is the volume measure of $A_{\alpha}$ for $\alpha<1$
(resp. arc length of $\Sc^{1}$ for $\alpha=1$).
Intuitively, \eqref{eq:covar func band lim der conv} means that there exists a {\em coupling} of the relevant fields, so that
the random fields \eqref{eq:fxT scaled def} converge, in a sense to be made precise below, to $\gfr_{\alpha}$.
Note that the random field $\gfr_{1}$ coincides with Berry's
Random Waves Model in \S\ref{sec:spher harm RMW}.

\vspace{2mm}

We follow the general strategy of Nazarov-Sodin ~\cite{So12,NS15} by first establishing the analogous
results for the limit random field $\gfr_{\alpha}$ (while separating the individual $k$) on $\R^{2}$,
then prove that these extend to $\M$, locally around every point $x\in \M$ (restricted to geodesic balls of
radius commensurable with $1/T$), and finally glue all the local results on $\M$ into a global one.
Counting tangencies poses a number of marked significant challenges as compared to the original setting
of Nazarov-Sodin, and also ~\cite{SW18}.

To pass from the Euclidean setting into the Riemannian one, a high probability ``stable" event is created
(``few" low lying critical points),
so that, if it occurs, ``most" of the nodal components cannot disappear or merge while perturbing the sample function,
and only few new ones can appear. It was noticed in ~\cite{SW18} that, on the stable event, not only the number of nodal components is (almost) preserved, but also their topologies. However, the number of $V$-tangencies of a component is {\em not} a topological
property, and so can vary upon an arbitrarily small perturbation of the sample function. We then have to redefine the stable
event to account for possible ``non-transversality" of the tangencies by translating our problem into one surrounding the quantitative transversal intersection between curves. We also have to use a somewhat
different analysis for proving that the new stable event is of high probability.

The reason why the support of $\mu_{\alpha}$ does not contain odd integers nor $0$ is that, in the Euclidean setup, if a direction $\zeta$ is fixed, for a closed simple curve to have $0$ or odd number of $\zeta$-tangencies would force some of the tangencies to be non-transversal, occurring with
probability $0$. Since the Riemannian setup will inherit the direction distribution measures from the Euclidean one, it will also induce the analogous results for the band-limited functions (though having a few nodal components with $0$ or odd number of $V$-tangencies is not dismissed a.s.). The proof for the support being exactly $2 \Z_+$ follows essentially from a $C^2$ analogue of the barrier method as utilized in \cite{NS15, SW18, CanSar18} in combination with a necessary $C^2$-closeness estimate for transversal intersections established in \S\ref{s:local_results}.

Another new challenge we encountered is that no analysis of the type explained above is possible around the
(finitely many) points $x\in \M$ where $V(x)=0$.
To deal with this situation we will excise small radius-$\rho$ balls around these problematic points, and bound
the contribution of their neighbourhoods to our counts, taking $\rho\rightarrow 0$ at the very end. Fortunately,
it is possible to tune all the parameters encountered and use their relations in our favour. Finally, the local computations with Kac-Rice are useful in order to establish that
the limit measure $\mu_{\alpha}$ is probability, as their means \eqref{eq:tot numb tang=sum mean} (or, rather, their Euclidean counterparts in
\eqref{eq:sum k*nod(k)<=Ncc} below) stay bounded.

\subsection{Discussion: further questions}

\subsubsection{Deterministic bounds}

Yau's conjecture ~\cite{Yau1,Yau2} states that if a function $\phi_{j}:\M\rightarrow\R$ satisfies \eqref{eq:Helmholts}
on a $d$-manifold $\M$,
then its nodal volume, i.e. the hypersurface volume of the nodal set $\phi_{j}^{-1}(0)$ is commensurable with $t_{j}$, i.e.
\begin{equation*}
t_{j}\ll_{\M}\vol(\phi_{j}^{-1}(0)) \ll_{\M} t_{j}.
\end{equation*}
Yau's conjecture was settled by Br\"{u}ning ~\cite{Br}, Br\"{u}ning-Gromes ~\cite{BrGr}, and Donnelly-Fefferman ~\cite{DF}
for real analytic manifolds (lower and upper bounds), and, more recently,
a breakthrough progress was made towards the general smooth case ~\cite{Log Mal,Log Lower, Log Upper}. It would be desirable to
establish the analogous deterministic bounds for the direction distribution of the type
\begin{equation*}
t_{j}^{2}\ll_{\M} \Ncc_{V}(\phi_{j}) \ll_{\M} t_{j}^{2}
\end{equation*}
for a $2$-manifold $\M$, and $V\in \mathcal{V}(\M)$, in some ``generic" scenario. Other than
the aforementioned bounds for the toral Laplace eigenfunctions \eqref{eq:fn ARW def}
due to Rudnick-Wigman ~\cite{RudWig2018}, appealing to B\'ezout's Theorem,
and the case of spherical harmonics
where one may exploit the fact that these are restrictions of polynomials and can therefore also employ B\'ezout, this problem is entirely open, at least to our knowledge.

\subsubsection{Uniform convergence w.r.t. $V$}

The convergence \eqref{eq:main meas conv} of the direction distribution measure to the limit measure is a priori dependent on $V$.
However, one could explicate the proofs to control the corresponding constants uniformly on compact sets of $\mathcal{V}(\M)$ with respect to some reasonable topology that takes into account $V$ and finitely many of its derivatives. One would be interested in characterising the geometry of such vector fields.

\subsubsection{Nonconstant vector fields on $\R^{2}$}

Our proofs below that $\mu_{\alpha}$ is independent of $V\in\mathcal{V}(\M)$ are based
on the fact that if $F:\R^{2}\rightarrow\R$ is a stationary isotropic Gaussian random field, then
the law of $\pazocal{N}_{\xi}(F|_{B(R)})$ (and $\mu_{F|_{B(R)}}(\xi)$), corresponding to the restriction of $F$ on a (large)
ball $B(R)$, is independent of the direction $\xi$. Here $\xi$ should be
thought as the direction corresponding to $V(x)$, $x\in\M$ is a fixed point on a Riemannian manifold,
and $F$ should be thought of a scaled version of a random field defined on $\M$, in the vicinity of $x$.
This naturally raises the question whether the same is true for $W\in\mathcal{V}(\R^{2})$ in place of $\xi$,
i.e. the distribution of $\pazocal{N}_{W}(F|_{B(R)})$ is independent of $W$, for $W$ that is no longer assumed to be constant,
at least, asymptotically as $R\rightarrow\infty$.

\subsection{Counting version of Theorem \ref{thm:main meas conv}}

Unlike in Theorem \ref{thm:main meas conv}, the following Theorem \ref{t:main_thm_less_equal_1} will separate the counts for different $k$.
A short argument will then allow to deduce Theorem \ref{thm:main meas conv} from Theorem \ref{t:main_thm_less_equal_1}
(see \S\ref{sec:proof main thm meas}).

\begin{theo} \label{t:main_thm_less_equal_1}
Let $V \in \mathcal{V}(\M)$ and $\alpha \in [0,1]$. Let $f_{\alpha, T}$ be the random $\alpha$-band limited functions \eqref{eq:fT band lim alpha<1} (or \eqref{eq:fT band lim alpha=1}) of degree $T$,
and $\nod_{V}(f_T,k)$ be the number of components whose number of $V$-tangencies is precisely $k$, $k\ge 0$. Then the following hold:
\begin{enumerate}
\item There exists $C_{\alpha, k} \geq 0$ such that
\begin{equation*}
\E \left[ \left| \frac{\nod_{V}(f_{\alpha, T},k)}{\nod(f_{\alpha, T})} - C_{\alpha, k} \right| \right]
\xrightarrow[T \rightarrow \infty]{} 0 .
\end{equation*}

\item
Furthermore, $C_{\alpha, k} > 0$ for $k \in 2 \Z_{> 0}$ and $C_{\alpha, k} = 0$ otherwise.
\end{enumerate}
\end{theo}

\section{Direction distribution for Euclidean random fields}

\subsection{Euclidean random fields: scale invariant model}

Throughout this article we will be interested in centered Gaussian random fields $F: \mathbb{R}^2 \rightarrow \R$.
By Kolmogorov's Theorem we know that the law of a centered Gaussian field is determined by its covariance kernel
\begin{equation*}
K(x,y) = \E \left[ F(x) \cdot F(y) \right].
\end{equation*}
In particular, we concern ourselves with those fields which are isotropic, that is,
\begin{equation*}
K(x,y) = K(|x-y|),
\end{equation*}
implying invariance under all rotations and translations. Furthermore, we assume for convenience that $K(0) = 1$.

It is known that such covariance kernels $K$ can be expressed as the Fourier transform of a measure $\rho$, called the \textit{spectral measure} of $F$. In many cases, it is more convenient to describe the field $F$ in terms of $\rho$ instead of the covariance kernel.  This alternative way of writing $F$ is through the Hilbert space $\mathcal{H}(\rho) = \mathcal{F} \left( L^2_{sym}(\rho) \right)$, that is the space of functions $h$ which are $L^2$-orthonormal with respect to the measure $\rho$ and satisfy the symmetry rule $\overline{h}(-Y) = h(Y)$.  In particular, if $\{\phi_k\}_k$ is an orthonormal basis for $\mathcal{H}(\rho)$, we can formally write
\begin{equation*}
F = \sum \zeta_k \phi_k
\end{equation*}
where $\zeta_k$  are i.i.d Gaussian random variables. This series diverges a.s. in $\mathcal{H}(\rho)$ but under suitable assumptions on $\rho$, converges a.s. pointwise or in some other sense to a well-defined function, the nature of the convergence depending on the properties of $\rho$.

For our purposes, the series of interest converges locally uniformly in $C^k$ for all $k \in \N$.
As a concrete example, the spectral measure $\rho$ corresponding to Berry's RWM mentioned in \S\ref{sec:spher harm RMW} as the scaling limit
of $T_{l}$ is the 1-dimensional arc-length measure $ds$ on the circle $S^1$, and we can explicitly compute an orthonormal basis for $\mathcal{H}(ds)$, so we can express it as
\begin{equation*}
\sum_{n \in \Z} c_n J_{|n|}(r) e^{in \theta}
\end{equation*}
where the $c_k$ are i.i.d. standard complex Gaussians.

\subsection{Kac-Rice method}

The Kac-Rice formula is a standard tool or a meta-theorem for expressing all the moments of {\em local} quantities, an example being the nodal volume, and the number of critical points of a random field $F:\Dc\rightarrow\R$ where
$\Dc$ might be a compact subdomain of the Euclidean space or a subdomain of $\M$. First, for our purposes
we need the following two upper bounds on the nodal count, which are direct conclusions from the application of Kac-Rice formula on appropriately defined critical points (either of a function $F$ or its restriction on the boundary of a disc or geodesic ball in $\M$)
appearing in ~\cite{BW17}, which were themselves borrowed from ~\cite{So12}. These results are valid for all random fields of our interest,
lying in the $2$ dimensional case only.

\begin{lemma}[Cf. {~\cite[Lemma 3]{BW17}}, {~\cite[Corollary 2.3]{So12}}]
\label{lem:Kac Rice Euclid}

Let $R>0$ and $F:\R^{2}\rightarrow\R$ be a stationary Gaussian random field, so that $F(\cdot)$ is a.s. $C^{2}$-smooth and so that for every $x\in \R^{2}$
the vector $\nabla F(x)$ is non-degenerate Gaussian.

\begin{enumerate}

\item Let $\nod(F,R)$ be the number of nodal components of $F$ lying entirely in $B(R)$. Then
\begin{equation*}
\E[\nod(F,R)] = \Oc(R^{2}),
\end{equation*}
with the constant involved in the $\Oc$-notation depending on the law of $F$ only.

\item Let $\widetilde{\nod}(F,R)$ be the number of nodal components of $F$ intersecting the circle $\partial B(R)$. Then
\begin{equation*}
\E[\widetilde{\nod}(F,R)] = \Oc(R),
\end{equation*}
with the constant involved in the $\Oc$-notation depending on the law of $F$ only.

\end{enumerate}

\end{lemma}

Next, the following result will be useful for ruling out mass leaking for the direction distribution measures (see the proof of
Theorem \ref{thm:euclid_count} (4) in \S\ref{sec:proof thm euclid no leak} below). It is a direct consequence of ~\cite[Theorem 6.3]{AW}.

\begin{lemma}
\label{lem:Kac-Rice Euclid tangencies}
Let $F:\R^{2}\rightarrow\R$ be a stationary Gaussian random field, so that the random vector $(F(0),\nabla F(0))$
has a non-degenerate $3$-variate Gaussian distribution, $\zeta\in S^{1}$,
and denote $\Ncc_{\zeta}(F,R)$ to be the number of $\zeta$-tangencies of $F$. Then
\begin{equation*}
\E[\Ncc_{\zeta}(F,R)] = C_{0}\cdot \vol(B(R)),
\end{equation*}
for some $C_{0}>0$ depending on the law of $F$ only.
\end{lemma}

Finally, the following lemma yields an upper bound for the number of nodal components lying in a small geodesic ball
for the band limited functions only.

\begin{lemma}[{~\cite[Lemma 2]{So12}}]
\label{lem:Kac Rice manifold}
Let $\alpha \in [0,1]$ and $f_{T}=f_{\alpha;T}$ be the random band-limited function \eqref{eq:fT band lim alpha<1}
(or \eqref{eq:fT band lim alpha=1}).
For $x\in \M$ and $r>0$ sufficiently small, denote $\nod (f_{x,T},r)$ to be the number of nodal components of $f_{T}$ lying inside
the geodesic ball $B(x,r)\subseteq\M$ centred at $x$ with radius $r$. Then their expected number satisfies the
following estimate:
\begin{equation*}
\E[\nod (f_{x,T},r)] = \Oc(T^{2}\cdot r^{2}),
\end{equation*}
with constant involved in the $\Oc$-notation depending on $\M$ and $\alpha$ only.
\end{lemma}

\subsection{Direction distribution for Euclidean random fields}

Let $F:\R^{2}\rightarrow\R$ be a smooth Gaussian random field with spectral measure $\rho$,
$\tilde{V} \in \mathcal{V}(\R^2)$ be non-vanishing, and $k \in \Z_{\geq 0}$. We denote by
\begin{equation*}
\nod_{\tilde{V}}(F,k,R)
\end{equation*}
(resp. $\nod(F,R)$) the number of connected components of $F^{-1}(0)$ (which are all smooth thanks to Bulinskaya's Lemma \cite[Proposition 6.11]{AW}) that are both completely contained in $B(R)$ and whose number of $\tilde{V}$-tangencies is precisely $k$
(resp. the total number of connected components of $F^{-1}$ contained in $B(R)$).

As we would like to study the distribution of $\nod_{\tilde{V}}(F, \cdot, \cdot)$, it is essential to show that this function is in fact a random variable (i.e. that it is measurable on the given sample space of Gaussian fields $F$). A detailed verification of this fact for the related random variable $N(F, R, \mathcal{T})$, that is the number of connected components contained in the ball $B(R)$ with topological type $\mathcal{T}$, was carried out in the Appendix of \cite{SW18} and is sufficiently robust to prove measurability in our case. We leave the details to the interested reader. We now state some axioms on $F$ (rather, its spectral measure $\rho$) so that to be able to formulate the main result (Theorem \ref{thm:euclid_count} immediately below)
on Euclidean random field.

\begin{definition}[Axioms on $\rho$]
\label{def:axiom spec meas}
For a given Gaussian stationary field $F$, we will sometimes impose the following restrictions on the corresponding spectral measure $\rho=\rho_{F}$:
\begin{itemize}
\item $(\rho 1)$ The measure $\rho$ has no atoms (if and only if the action of the translations is ergodic by Theorem \ref{thm:GFM}).
\item $(\rho 2)$ For some $p > 4$,
\begin{equation*}
\int\limits_{\R^2} |\lambda|^{p} \, d \rho(\lambda) < \infty
\end{equation*}
(this ensures that a.s. $C^2$-smoothness of $F$).
\item $(\rho 3)$ The support of $\rho$ does not lie in a linear hyperplane (this ensures that the Gaussian field, together with its gradient, is not degenerate).

\item $(\rho 4^*)$ The support of $\rho$ has non-empty interior.

\end{itemize}
\end{definition}

Nazarov and Sodin ~\cite{So12,NS15}
proved that if $F$ satisfies the axioms $(\rho 1)-(\rho 3)$, then there exists a constant $c(\rho)=c_{NS}(\rho)\ge 0$
(``Nazarov-Sodin constant "), so that
\begin{equation}
\label{eq:NS const def prop}
\E\left[\left|\frac{\nod(F,R)}{\vol(B(R))} - c(\rho) \right|\right] \rightarrow 0.
\end{equation}
If, further, $(\rho 4^{*})$ is satisfied, then $c(\rho)>0$. In fact, in most of what follows, we will work with the following weaker version
of $(\rho 4^{*})$.

\begin{definition}[Axiom $(\rho 4)$]

$(\rho 4)$ The Nazarov-Sodin constant $c(\rho)$ is positive.

\end{definition}

\begin{theo} \label{thm:euclid_count}
Let $F: \R^n \rightarrow \R$ be a stationary random field whose spectral measure $\rho$ satisfies axioms $(\rho 1) - (\rho 4)$,
$k \ge 0$, $\zeta\in\R^{2}$ be a fixed direction, and recall the Nazarov-Sodin constant $c(\rho)>0$ satisfying the defining
property \eqref{eq:NS const def prop}. Then:
\begin{enumerate}
\item There exists $C_{k, \zeta}=C_{k, \zeta}(\rho) \geq 0$
\begin{equation}
\label{eq:conv L1 Euclid}
\E \left[ \left| \frac{\nod_{\zeta}(F,k,R)}{c(\rho)\cdot \vol (B(R))} - C_{k, \zeta} \right| \right] \rightarrow 0
\end{equation}
as $R \rightarrow \infty$. If $k=0$ or $k$ is odd, then $C_{k,\zeta}=0$.

\item Assuming further that $F$ is isotropic, we have that $C_k=C_{k, \zeta}$ is independent of $\zeta$.

\item If $k\in 2\Z_{>0}$, and either $F$ satisfies the axiom $(\rho 4^*)$ (without assuming that $F$ is isotropic),
or $\rho = \sigma_{S^1}$ (i.e. $F$ is Berry's monochromatic isotropic random waves),
the normalized arc-length measure on $S^1$, then $C_{k,\zeta} > 0$.

\item We have
\begin{equation*}
\sum\limits_{k=0}^{\infty}C_{k,\zeta} = 1.
\end{equation*}

\end{enumerate}
\end{theo}

\section{Various facts from plane geometry}

\subsection{Vector fields and exponential maps}

\begin{defn} \label{d:blown_up_V}
Let $F  \in C^{\infty}(T_x\M)$ and $T>0$. For $V \in \mathcal{V}(\M)$ and $Y \in T_x\M$, we define the \textit{blown-up (at scale $T$) vector field} $\tilde{V}_{x,T} \in \mathcal{V}(T_x\M)$ at $x \in \M$, where $exp_x(0) = x$ and $exp_x(Y)=y$, via the local coordinate formula
\begin{equation}
\label{eq:tild(V) def}
(\tilde{V}_{x,T} F) =
 a_1 \left( \exp_x \left( \frac{Y_1}{T}, \frac{Y_2}{T} \right) \right) \cdot \frac{\partial F}{\partial Y_1}\left( Y  \right)
+ a_2 \left( \exp_x \left( \frac{Y_1}{T}, \frac{Y_2}{T} \right) \right) \cdot \frac{\partial F}{\partial Y_2} \left( Y \right).
\end{equation}
\end{defn}

We will shortly consider the random fields
$F(Y_1, Y_2) = f_{\alpha, T}(\exp_x(\frac{Y}{T}))$ as in \eqref{eq:fxT scaled def}. Notice that we are pushing forward the vector field $V$, extracting the ``top order in $T^{-1}$" differential operator, and then applying that to $F$. We {\em are not} pushing forward $V$ by $M_{1/T} \circ (\exp_x)^{-1}_*$, where $M_{1/T}$ is simply multiplication in the fibers of $T_xM$ by $1/T$. The reason for this particular definition is so that the coefficients of $V$ are localizing at the same scale as $f_{x, T}$.
Thus in the regime $|Y| \leq R$ with $R/T = o(1)$, for a large but fixed $R$ with $T \rightarrow \infty$, we find that the coefficients $a_1,a_2$ become constants up to first order in $T^{-1}$. More specifically, we have that, for instance, $a_1 \left( \exp_x(\frac{Y_1}{T}, \frac{Y_2}{T}) \right) = a_1(x) + \mathcal{O} \left( \frac{|Y|}{T} \right)$ with $\frac{|Y|}{T} = o(1)$ as seen in upcoming sections discussing limiting regimes.

From here onward, we reserve the notation of $\tilde{V}$ for vector fields on $\R^2$. Now, in order to show that our main result Theorem \ref{t:main_thm_less_equal_1} is not vacuous, we give the following result on the existence of vector fields with finitely many zeroes.

\begin{lem} \label{l:vfs_finite_zeroes}
Given a smooth and compact surface $(\M,g)$, there exists smooth vector fields $V$ that have finitely many zeroes. That is, our class $\mathcal{V}(\M) \neq \emptyset$ for compact Riemannian surfaces $\M$.
\end{lem}

This result holds in $d\ge 2$ dimensions, but we state it for surfaces given our context.

\begin{proof}
Consider any Morse function $f$ on $\M$; we know there exists such a function, for example one being the height function on $\M$ after applying an appropriate embedding into $\R^N$. The metric $g$ allows us to associate the differential $df$ to its gradient field $\nabla_g f$, each of whose zeroes is isolated by the non-degenerate critical point condition. Compactness of $\M$ yields that the number of such zeroes must be finite.
\end{proof}

\subsection{Global plane geometry}

In this section we begin our study of curves in $\R^2$ that have a fixed number of tangencies to a given vector field $\tilde{V}$ and its relation to the study of transversality between the sets $f^{-1}(0)$ and $(Vf)^{-1}(0)$. These elementary results, although simple to state and are intuitively clear, prove very useful in the pursuit of Theorem \ref{t:main_thm_less_equal_1}. We begin with a simple but crucial lemma that will play a key role in showing that the limiting constants $C_k$, for $k=0,1$, in Theorem \ref{t:main_thm_less_equal_1} are $0$:

\begin{lem} \label{l:at_least_two_tangencies}
Suppose $\zeta$ is a fixed direction on $\R^2$. Then any closed simple curve in the plane must have at least two tangencies $\zeta$.
\end{lem}

Before proving this, we remind ourselves of a simple form of the Gauss-Bonnet Index Theorem, which goes by the name of the Theorem of Turning Tangents in elementary differential geometry, due to Hopf:

\begin{theo}[Theorem of Turning Tangents] \label{t:turn_tang}
The rotation number of an embedded $S^1$ in $\R^2$ (i.e. the oriented number of times the unit tangent vector $\overrightarrow{T}$ to this embedding transverses $S^1$) is equal to $\pm 1$.

In other words, the map induced by $\overrightarrow{T}$ from $S^1$ to $S^1$ must be surjective and in terms of topological language, $\overrightarrow{T}$ induces a map of degree $\pm 1$.
\end{theo}

\begin{proof}[Proof of Lemma \ref{l:at_least_two_tangencies}]
By Theorem \ref{t:turn_tang}, $\overrightarrow{T}$ must be parallel to $\zeta$ and -$\zeta$ at at least two distinct points.  Since $\overrightarrow{T}$ parallel to -$\zeta$ is counted as a $\zeta$-tangency, this proves our claim.
\end{proof}

It is possible to construct closed curves in $\R^2$ with an odd number of tangencies:
all we must do is properly glue together neighbourhoods of the singular parts of a cubic to the neighbourhoods of singular parts of the quadratics. However, we will see that this situation is unstable in an appropriate sense, and therefore will not contribute towards our final result.

\begin{lem} \label{l:even_int}
Consider two $C^2$ curves $\gamma_1, \gamma_2 \subsetneq \R^2$ that arise as components of the zero sets of respectively two different functions, both of which has 0 as a regular value. Assume that $\# ( \gamma_1 \cap \gamma_2)  < \infty$.  If $\gamma_1$ is both simple and closed or $\gamma_2$ is both simple and closed, and each point in $\gamma_1 \cap \gamma_2$ is transversal, then $\# \gamma_1 \cap \gamma_2$ is even.
\end{lem}

\begin{proof}
Without loss of generality, let $\gamma_1$ be simple and closed.  By the Jordan Curve Theorem, $\gamma_1$ encloses a bounded region $\Omega_1$.  By the transversality assumption, every point of intersection is uniquely connected to another intersection  by an arc  of $\gamma_2$ that is completely contained in $\Omega_1$ (for otherwise, this contradicts the regularity assumption or the transversality assumption).  Hence, we can associate two unique points of intersection to each interior arc.  Counting the number of such arcs, we arrive at our conclusion.
\end{proof}

\begin{cor}
\label{cor:tangencies even}
For a fixed direction $\zeta$ on $\R^2$, any simple closed curve in the plane without inflection points that are also $\zeta$-tangent must have an even positive number of $\zeta$-tangencies.
\end{cor}

We close this section by emphasizing that by the qualitative version of Bulinskaya's Lemma (see \cite[Proposition 6.11]{AW}),
with probability 1 any nodal component of a Gaussian field $F$ on $B(0,R)$ is simple and regular (that is 0 is regular value for $F$) although there may exists those that are not closed.

\subsection{Intersections in the plane: local results} \label{s:local_results}

The terminology established in the section is important from both analytic and geometric points of view. It is here where we establish our main observation: \textit{the study of nodal curves with a fixed number of $V$-tangencies is equivalent to the study of the transversal and tangential intersections of the sets $f_{T, \alpha}^{-1}(0)$ and $(Vf_{T, \alpha})^{-1}(0)$.}

Our aim is to interchange the words of ``intersection" and ``$V$-tangency" with the latter being a special case of the former in the upcoming dictionary. Although this may appear counter-intuitive at first, it is crucial that we identify our geometric property with a suitable analytic characterization.

\subsubsection{Deterministic intersection results: $\beta$-transversal tangencies} \label{s:deterministic}

\begin{defn} \label{d:beta_trans}
\begin{enumerate}
\item For a pair of functions $G=(g_1, g_2): \R^2 \rightarrow \R^2$ such that $g_1, g_2 \in C^1(B(2R))$, $0$
is a regular value for both $g_1$ and
$g_2$, and $\beta>0$ (possibly small), we say that $u_0$ is a point of \textit{$\beta$-transverse intersection} if $g_1(u_0)=g_2(u_0) = 0$ and $\left| \det DG(u_0) \right| >  \beta$.

\item Let $\beta>0$ be a small parameter.
Given $g \in C^2(\R^2)$ such that $0$ is a regular value with $\tilde{V}$ a smooth vector field on $\R^2$, we say that $u_0$ is a \textit{$\beta$-transverse $\tilde{V}$-tangency} (or simply a ``non-degenerate" tangency) if $
\tilde{V}(u_0) \neq 0 \in \R^2$, $g(u_0)=\tilde{V}g(u_0)=0$, and $\left| \det \left( D_{u=u_0} \left[ \nabla g, \nabla \tilde{V}g \right] \right) \right| > \beta$.
\end{enumerate}
\end{defn}

Now that we have fixed some useful terminology, we proceed to establishing some key yet general estimates on the intersections of curves in the planes and their behaviors under perturbations.  In this section, we analyse a situation of quantitative transversality:

\begin{prop}[Deterministic stability of transversal intersections, cf. {\cite[Lemma 10]{BW17}}] \label{p:stab_trans}
Let $R>0$ and $k \ge 0$ be given. Let $\tilde{V} \in \mathcal{V}(\R^2)$ be non-vanishing, $g,h \in C^3(\overline{B}(2R))$, and denote $G:=(g, \tilde{V}g)$. Assume that the following restrictions hold:
\begin{enumerate}
\item There exists $\beta_1 >0$ such that
\begin{equation*}
\min_{\overline{B}(2R)} \max \{ |g(u)|, |\nabla g(u)| \}, \,  \min_{\overline{B}(2R)} \max \{ |\tilde{V}g(u)|, |\nabla (\tilde{V} g)(u)| \} > \beta_1.
\end{equation*}
This immediately implies that
\begin{equation*}
\min_{\overline{B}(2R)} \max \{ \|G\|, \|DG\|_{\mathcal{L}(\R^2 \rightarrow \R^2)} \} > \frac{\beta_1}{\sqrt{2}},
\end{equation*}
and in particular that $\|g\|_{C^2} > c_2\beta_1$ where $c_2>0$ is a constant depending only on the ambient dimension.
\item There exists $\beta_2$ such that
\begin{equation*}
\left| \det DG(u) \right| >  \beta_2
\end{equation*}
for $u \in \{ (g, \tilde{V}g) = 0 \}$ in $B(2R)$.
\item The $C^3$ norms of $g,h$ are bounded above by $M_0$ for some $M_{0}>0$.
\item We have
\begin{equation*}
\| g - h \|_{C^2(B(R+1))} < b_0
\end{equation*}
for some $b_{0}=b_{0}(M_{0},\beta_{0},\beta_{1})>0$ sufficiently small.

\end{enumerate}

Then there exists an injective map $\gamma \rightarrow \gamma^h$ between the components of $g^{-1}(0)$ and $h^{-1}(0)$ with the property that:
\begin{enumerate}
\item The connected components $\gamma \subset B(R-1)$ of $g^{-1}(0)$ whose number of $\beta$-transversal $\tilde{V}$-tangencies are precisely
$k$ is in one-to-one correspondence with a subset of the connected components $\gamma^{h} \subset B(R)$ of $h^{-1}(0)$ whose number of $(\beta-b)$-transversal $\tilde{V}$-tangencies is precisely $k$.

\item For every $\gamma$ as above, the components $\gamma$ and $\gamma^h$ are uniformly close in the following sense: there exists a smooth bijective map $\psi_{\gamma}: \gamma \rightarrow \gamma^h$ such that for all $u \in \gamma$, we have the estimate
\begin{equation} \label{e:components_bouge}
\| \psi_{\gamma}(u) - u \|_{L^{\infty}(B(R))} = \mathcal{O}(b)
\end{equation}
where the implicit constant depends on the quantities $\|G\|_{C^2(B(0,2R))}$ and $\beta_1$.
\end{enumerate}
\end{prop}

The dependence of the implicit constant in the estimate \eqref{e:components_bouge} for the new zero's positioning \textit{linear} in $\beta_1$ but of order $-1$ in $\beta_2$.

\subsubsection{Some technical lemmas}

In this section we give a key lemma following from multivariable calculus, which is essentially a restatement of the above proposition, and is analogous to \cite[Lemma 7]{So12}.

\begin{lem} \label{l:Sodin_type_calc_lem}
Consider $G(u) = (g_1(u), g_2(u))$ for $g_1, g_2 \in C^2(B(2R)$ for some given $R>0$ possibly large. Suppose $u_0 \in G^{-1}(0) \cap B(2R)$. We make the following assumptions:
\begin{itemize}
\item There exists $1 > a>0$ such that for all $u \in B(2R)$, we have
\begin{equation} \label{e:no_low_lying_cp}
\min_{\overline{B}(2R)} \max \{ |g_1(u)|, |\nabla g_1(u)| \}, \,  \min_{\overline{B}(2R)} \max \{ |g_2(u)|, |g_2(u)| \} > a.
\end{equation}
\item There exists $\alpha> 0$ such that $\alpha < |\det DG(u)|$ for all $u \in B(2 R) \cap G^{-1}(0)$.
\item Set
\begin{equation} \label{e:perturb_size}
 b =\frac{1}{4} \min \left\{ a , \frac{1}{ \|G\|_{C^2(B(u_0, R_0+1))}}, \alpha \right\}.
\end{equation}

\end{itemize}
Let $\tilde{V} \in \mathcal{V}(\R^2)$ be non-vanishing.  Consider a perturbation $\Psi = (\psi, \tilde{V} \psi)$ where $\psi \in C^2(B(2R))$.  Set $G_{\Psi}(u) = G(u) + \Psi(u)$ and consider only zeroes $u_0 \in G^{-1}(0)$ such that $B(u_0, b) \subset B(0, R_0)$.  Then for $\| \psi \|_{C^2(B(2R))} < b$ we have that
\begin{enumerate}
\item $u_0$ is a isolated zero of $G$,
\item $G_{\Psi}$ is a global diffeomorphism on $B ( u_0,  b )$,
\item $G_{\Psi}$ also has a unique isolated zero $\tilde{u_0} \in B(u_0, b) $ with
\begin{equation*} \label{e:new_zero_pos}
\|u_0 - \tilde{u_0} \| =  \mathcal{O}(b)
\end{equation*}
where the implicit constant depends on the quantities $\|G\|_{C^2}$ and $a^{-1}$.
Moreover, the new zero $\tilde{u_0}$ has the property that $|det DG_{\psi}(\tilde{u_0})| > \alpha - b.$
\end{enumerate}
\end{lem}

\begin{proof}
Each of these claims follows almost directly from the Inverse Function theorem and its proof. Part (1) is an immediate implication of the inverse function theorem.

Parts (2) and (3) use elements from the proof of the Inverse Function theorem. Here, we closely follow \cite[Theorem 22.26]{B03}. If we can show that the so-called \textit{Newton's function} $\epsilon_{\Psi}(u) = DG_{\Psi}^{-1}(u_0)G_{\Psi}(u+u_0) - u$ has the property $\|D\epsilon_{\Psi}(u)\|  < 1$ for all $u \in B(0, b)$ with $b$ satisfying the stated inequality, the referenced proof of the Inverse Function theorem shows that $\epsilon_{\Psi}$ is invertible on $B(0, b)$ and therefore $G_{\Psi}$ is invertible on $B(u_0, b)$. Thus, we estimate the operator norm for the matrix $D \epsilon_{\Psi}$:
\begin{align*}
& \left( DG(u_0) + D\Psi(u_0) \right)^{-1} \left( DG(u+u_0) + D\Psi(u+u_0)  \right) - I \\
& = \left( DG(u_0) + D\Psi(u_0) \right)^{-1} \left( DG(u_0) + D\Psi(u+u_0) + \mathcal{O}_{\|G\|_{C^2(B(2R))}}(\|u\|)  \right) - I \\
&  = \left(I + DG(u_0)^{-1} D\Psi(u_0) \right)^{-1} (DG(u_0))^{-1} \left( DG(u_0) + D\Psi(u+u_0) +  \mathcal{O}_{\|G\|_{C^2(B(2R))}}(\|u\|) \right) - I  \\
& = \left( I + DG(u_0)^{-1} D\Psi(u_0) \right)^{-1} \left( I + DG(u_0)^{-1} D\Psi(u+u_0) + \mathcal{O}_{ \|G\|_{C^2(B(2R))}}(a^{-1} b )  \right) - I \\
&= \mathcal{O}\left(a^{-1} \cdot b \right) +  \mathcal{O}\left( \|G\|_{C^2(B(u_0, b))} \cdot a^{-1} \cdot b \right) +  \mathcal{O} \left((a^{-1}b)^2 \right) \\
& = \mathcal{O}_{a, \|G\|_{C^2}}(b),
\end{align*}
provided we have the restriction \eqref{e:perturb_size} and $a < 1$, where we have used some well-known formulas for inverses of small perturbations of the identity matrix. This in turn implies that the inverse function $G_{\psi}^{-1}$ exists on $B(u_0, b)$, therefore establishing claim (2).  Hence any zero, if it exists, lying inside of $B(u_0, b)$ must be unique and isolated.  We have now reduced ourselves to proving the existence of a new zero $\widetilde{u_0}$.

The proof of the existence and exact positioning of the isolated zero is motivated by the proof of \cite[Lemma 10]{BW17}.
As $N_i(u_0) = \frac{\nabla g_i(u_0)}{\| \nabla g_i(u_0)\|}$ is well-defined thanks to $|\det DG(u_0) | > \alpha$, we consider the function
\begin{equation*}
\zeta(s_1, s_2) := G_{\psi}\left(u_0 + s_1 N_1(u_0) + s_2 N_2(u_0) \right);
\end{equation*}
we would like to show that $\zeta$ does obtain a zero $\tilde{s}_0$ and give an estimate on $\|\tilde{s}_0\|$. This would establish claim (3).
To begin, consider the auxiliary function
\begin{equation*}
\mathfrak{S}_{(t_1, t_2)}(s_1, s_2) := (t_1, t_2) - \epsilon_{\zeta}(s_1,s_2)
\end{equation*}
where $\epsilon_{\zeta}(s_1, s_2) = D\zeta^{-1}(0,0) \zeta(s_1,s_2) - (s_1, s_2)$ and $(t_1,t_2)$ are small numbers; we will employ an effective version of the Contraction Mapping Theorem in order to determine the existence and size of $\tilde{s}_0$.

Let us take a pause to understand the purposes of our various newly-defined auxiliary functions.  Supposing that $\epsilon_{\zeta}(s_1, s_2)$ were well-defined and that we were to able solve the fixed point equation
\begin{equation} \label{e:fixed_point_H}
\mathfrak{S}_{(t_1,t_2)} (H(t_1,t_2)) = H(t_1,t_2)
\end{equation}
for $(t_1, t_2)$ varying over some small neighborhood around the origin. This would imply that $H(0,0) = - \epsilon_{\zeta}(H(0,0))$. As $D\zeta^{-1}(0,0)$ is invertible our assumption $\epsilon_{\zeta}$ being well-defined, then this along with the equation $H(0,0) = - \epsilon_{\zeta}(H(0,0))$ would show that $ D\zeta^{-1}(0,0) \zeta(H(0,0))=(0,0)$.  Thus $ \zeta(H(0,0))=(0,0)$ and moreover $\zeta$ obtains a zero.  If we can show that  $$\|D_{s_1,s_2} \mathfrak{S}_{(t_1,t_2)}(s_1, s_2) \| < 1$$ for suitably small values of $(t_1,t_2)$ and $(s_1,s_2)$, then the Contraction Mapping Theorem as used in the proof of the Inverse Function Theorem gives us such an $H(t_1,t_2)$ as in (\ref{e:fixed_point_H}) and therefore $\tilde{s}_0:=H(0,0)$ is achieved.

To sum up our reasoning, first we must estimate the operator norm of $D_{s_1, s_2}\mathfrak{S}_{(t_1, t_2)}$ for $\|(t_1,t_2)\|, \|(s_1,s_2)\| \leq Cb$ to show the existence of $H(s_1,s_2)$ which in turn gives the existence of a zero $\tilde{s}_0 = H(0,0)$. Second, an estimate on $\|\epsilon_{\zeta}(s_1,s_2)\|$ for $\|(s_1, s_2)\| \leq C b$ where $C>0$ is some uniform constant gives us the positioning of $\tilde{s}_0$; this follows from the equation $H(0,0) = - \epsilon_{\zeta}(H(0,0))$. We now proceed to carrying out these steps.

A direct calculation will show that
\begin{equation*}
D_{s_1,s_2} \zeta_{|(0,0)} = D_{s_1,s_2} G_{\Psi} \left( u_0 + s_1 N_1(u_0) + s_2 N_2(u_0)\right)_{|(0,0)}
\end{equation*}
is invertible. Thanks to $G$ being a local diffeomorphism in $B(u_0, b)$ (and therefore $\{N_1, N_2\}$ forming a linearly independent set over $B(u_0, b)$), this combined with  $|\det (DG + D\psi)(u)| \geq \alpha - 2b > 0$, which follows by (\ref{e:no_low_lying_cp}) and (\ref{e:perturb_size}), for all $u \in B(u_0, b)$ finally shows that $D\zeta(0,0)$ is invertible. Thus, we have an expression
\begin{align*}
&  D_{s_1,s_2} \mathfrak{S}_{(t_1,t_2)}(s_1, s_2) = -D \epsilon_{\zeta}(s_1,s_2) \\
& = \left[N_1(u_0), N_2(u_0) \right]^{-1} \circ DG^{-1}_{\psi}(u_0) \circ DG_{\psi}(u_0 + sN_1 + sN_2) \circ \left[N_1(u_0), N_2(u_0) \right] - I
\end{align*}
where we arrive at a similar Taylor expansion analysis as carried out in the first half of our proof. We leave it to the interested reader to see that for $\|D_{s_1,s_2} \mathfrak{S}_{(s_1,s_2)}(s_1, s_2)\| = \mathcal{O}(b) < 1$, where the implicit constant again depends on $\|G\|_{C^2}$ and $a^{-1}$.  Another application of the Contraction Mapping Theorem, thanks to the operator norm of $D_{s_1,s_2} \mathfrak{S}_{(t_1,t_2)}(s_1, s_2)$ being smaller than 1, brings us to the conclusion that such a solution $H$ exists.

We know that $\epsilon_{\zeta}(s_1, s_2) = \epsilon_{\zeta}(0, 0) + \mathcal{O}_{\|D\epsilon_{\zeta}(0,0)\|}(b)$. Furthermore,
$\epsilon_{\zeta}(0,0) = D \zeta^{-1}(0,0) \Psi(u_0)$ since $u_0$ is a zero of $G$. Finally as $\|\psi\|_{C^2} < b$, we conclude that $\tilde{s}_0 = H(0,0) =  \mathcal{O}(b)$ with the corresponding implicit constants.
\end{proof}

Note that our application does require $C^2$ regularity due to setting $g_2 = \tilde{V}_{x,T}g_1$. Our proof demonstrates that the amount in which our zeroes move after perturbation is controlled by the $C^1$ norm of the perturbation $\Psi$ (thanks to our perturbation being simply additive in nature), emphasizing that from the point of view of analysing zeroes the $C^1$ approximation of the scaling limit $\mathfrak{g}_{\alpha}$ is sufficient.

\begin{proof}[Proof of Proposition \ref{p:stab_trans}]
Set $g = g_1$ and $\tilde{V}g = g_2$. This demonstrates why $C^2$ estimates are imposed. Now, we simply set $\beta_1 = a$, $\beta_2 = \alpha$, and take $b$ equal to the strength of the perturbation. The final step involves applying in tandem \cite[Lemma 10]{BW17} (although we can easily repeat the steps of the proof of Lemma \ref{l:Sodin_type_calc_lem} to recover this result) and conclusion (3) of Lemma \ref{l:Sodin_type_calc_lem} to show that non-degenerate components map to a subset of non-degenerate components and stay $\mathcal{O}(b)$ close in $L^{\infty}$.
\end{proof}

A combination the Proposition \ref{p:stab_trans} and \cite[Lemma 7]{So12} immediately gives us that:

\begin{cor} \label{l:Sodin_annuli}
Under the hypotheses of Lemma \ref{l:Sodin_type_calc_lem}, let $C>0$ be the implicit constant that appears in the estimate (\ref{e:new_zero_pos}).  Take $\Gamma_g$ a regular component of $g^{-1}(0)$. Then, if given a tubular neighborhood $A_{\Gamma_g}(Cb)$ that uniquely contains the component $\Gamma_g \subset g^{-1}(0)$ and has width $2Cb$, it follows that the points of $\{ (g, \tilde{V}g) = 0 \}$ which meet $\Gamma_g$ persist inside of $A_{\Gamma_g}(Cb)$ after perturbation by $h-g$.
\end{cor}

\subsubsection{Deterministic intersection results: sub-$\beta$ tangencies}

The main purpose of this section is to provide the reader with the intuition that components with ``near-tangential" intersections can be unstable in terms of intersection counts being constant when subjected to small perturbations. The following lemma will not be employed in the proofs of our main results, however, it serves as an intuition what might happen to the tangency points if we do not exclude the unstable event $\Delta_7$, as
defined in \S\ref{s:exceptional_events} below.

\begin{defn} \label{d:sub_beta}
\begin{enumerate}
\item For a function $G=(g_1, g_2)$ such that $g_1, g_2 \in C^1(B(2R))$ and $\beta>0$ (possibly small), we say that $u_0$ is a point of \textit{sub-$\beta$ intersection} if $G(u_0) =  (g_1(u_0), g_2(u_0)) = 0$ and $\left| \det DG(u_0) \right| \leq  \beta$.

\item Let $\beta>0$ be a small parameter. Given $g \in C^2(\R^2)$ with $\tilde{V} \in \mathcal{V}(\R^2)$ non-vanishing, we say that $u_0$ is a point of \textit{sub-$\beta$ $\tilde{V}$-tangency} (or simply a ``near-degenerate" tangency) if $( g(u_0), \tilde{V}g(u_0)) = 0$ and $|\det \left( D_{u=u_0} \left[ \nabla g, \nabla \tilde{V}g \right] \right) | \leq \beta$.
\end{enumerate}
\end{defn}

\begin{lem}[Classification of perturbations of sub-$\beta$ intersections] \label{l:class_tang}
Let $r>0$ and $\beta>0$ be given. Let $0$ be a regular value of $g_1, g_2 \in C^{\infty}(\R^2)$. Furthermore, we assume that $g_1^{-1}(0) \cap g_2^{-1}(0)$ is regular, i.e. that $|\det(\nabla g_1(u), \nabla g_2(u))| \neq 0$ for each point $u \in g_1^{-1}(0) \cap g_2^{-1}(0)$.  Suppose $u_0 \in B(u_0, r) \cap g_1^{-1}(0) \cap g_2^{-1}(0) \subset B(R)$ is an isolated point of ``near"-tangential intersection, that is
\begin{equation*}
|\det(\nabla g_1(u_0), \nabla g_2(u_0))| \leq \beta.
\end{equation*}

Then there exists $b(r)>0$ with $b < r$, such that for perturbations $\psi_1, \psi_2$ with $\|\psi_i\|_{C^2(B(R)} < b$ for $i=1,2$, the only possibilities of zeroes of $(g_1 + \psi_1)^{-1}(0)  \cap (g_2 + \psi_2)^{-1}(0)$ appearing inside of $B(u_0, r + b)$ are: 1) two sub-$(\beta+b)$ zeroes, that is we have $|\det(\nabla (g_1 + \psi_1), \nabla (g_2 + \psi_2))| < \beta + b$ on $B(u_i, r + b)$ for $i=1,2$ being the new zeroes, or 2) a single sub-$(\beta + b)$ zero in $B(u_0, r + b)$.
\end{lem}

\begin{proof}
By an elementary notion from differential topology \cite[Chapter 2]{GP}, if $K$ is a compact set and $f: K \rightarrow \R^2$ is transverse to a closed submanifold $L \subsetneq \R^2$, then the transversality is stable under suitable perturbations (i.e. of class $C^1$) of $f$.
Considering this fact, let $f$ be the parametrization of the portion of the nodal component of $g_1^{-1}(0)$ inside of $B(u_0, r)$ and $L$ the nodal component of $g_2^{-1}(0)$ inside of $B(u_0,r)$. If our initial choice of $b$ is small enough, then the transversality will persist. For a possibly smaller choice of $b$ depending on $t$, given value of $r$, this then yields 1) or 2).
\end{proof}

\section{Proof of Theorem \ref{thm:euclid_count}}

\subsection{Some preliminaries}

\begin{theo}[Grenander-Fomin-Maruyama] \label{thm:GFM}
Let $F$ be a random stationary Gaussian field with spectral measure $\rho$. Then, if $\rho$ contains no atoms, the action of the translations group
\begin{equation*}
(\tau_uF)(x) = F(x - u)
\end{equation*}
is ergodic ($F$ is ``ergodic").
\end{theo}

\begin{theo}[Wiener's Ergodic Theorem] \label{thm:weiner}
Suppose that $F$ is ergodic and the translation map $\R^n \times C(\R^n) \rightarrow C(\R^n)$
\begin{equation*}
(u,F) \mapsto \tau_u F
\end{equation*}
is measurable with respect to the product $\sigma$-algebra $\mathcal{B}(\R^n) \times \mathcal{A}$ and $\mathcal{A}$.  Then every random variable $\Phi(F)$ with finite expectation $\E \left[ |\Phi(F)| \right] < \infty$ satisfies
\begin{equation*}
\lim\limits_{R \rightarrow \infty} \frac{1}{\vol (B(R))} \int\limits_{B(R)} \Phi(\tau_u F) \, du \rightarrow \E \left[ \Phi(F) \right],
\end{equation*}
both almost surely and in $L^1$.
\end{theo}

\begin{defn}
\label{def:counts}
Let $\Gamma \subset \R^2$ be a smooth curve (connected or not) and $\zeta$ be a fixed direction. For $u \in \R^2$, $r >0$, and $k \ge 0$, we denote by
\begin{equation*}
\nod_{\zeta}(\Gamma, k, u, r)
\end{equation*}
the number of connected components of $\Gamma$ whose number of $\zeta$-tangencies, that are also lying entirely inside of the ball $B(u,r)$, is precisely $k$.  In the case of $u=0$, we simply write $\nod_{\zeta}(\Gamma, k, r)$.  Similarly, we define
\begin{equation*}
\nod^*_{\zeta}(\Gamma, k, u, r)
\end{equation*}
by relaxing the condition that the components of $\Gamma$ merely intersect $B(u,r)$.
\end{defn}

\begin{lem}[Integral-Geometric Sandwich] \label{l:integro_geom_sandwich}
Let $\Gamma$ be a smooth, closed curve and $\zeta$ be a fixed direction. Then for $0 < r < R$, $k \ge 0$, we have that
\begin{align*}
\frac{1}{\vol(B(r))} \int\limits_{B(R-r)} \nod_{\zeta}(\Gamma, k, u, r) \, du &\leq \nod_{\zeta}(\Gamma, k, R) \\
& \leq \frac{1}{\vol(B(r))} \int\limits_{B(R+r)} \nod^*_{\zeta}(\Gamma, k, u, r) \, du
\end{align*}
\end{lem}

\begin{proof}
Our proof follows almost exactly the steps of \cite[Lemma 1]{So12} and \cite[Lemma 5]{BW17}. Here, we present only the proof for the upper bound of $\nod_{\zeta}(\Gamma, k, R)$.

Let $\gamma \subset \Gamma$ be a connected component.  Define
\begin{equation*}
G_*(\gamma) = \bigcap_{v \in \gamma} B(v,r) = \{ u \, : \, \gamma \subset B(u,r) \}
\end{equation*}
and
\begin{equation*}
G^*(\gamma) = \bigcup_{v \in \gamma} B(v,r) = \{ u \, : \, \gamma \cap B(u,r) \neq \emptyset \}.
\end{equation*}
It follows by the above definition that for all $v \in \gamma$,
\begin{equation*}
G_*(\gamma) \subset B(v,r) \subset G^*(\gamma),
\end{equation*}
which in turn immediately yields the inequality
\begin{equation*}
\vol(G_*(\gamma)) \leq \vol(B(v,r)) \leq \vol(G^*(\gamma))
\end{equation*}
Finally, a sum over the components $\gamma \subset B(R)$ whose number of $\zeta$-tangencies is precisely k gives us
\begin{align*}
\sum_{\left\{\substack{\gamma \subset B(R), \\ \gamma \mbox{ has } k \mbox{ tangencies}}\right\}} \vol(G_*(\gamma)) & \leq \vol(B(r)) \, \nod_{\zeta}(\Gamma,k, R) \\
& \leq \sum_{\left\{ \substack{\gamma \subset B(R), \\ \gamma \mbox{ has } k \mbox{ tangencies}}\right\}} \vol(G^*(\gamma))
\end{align*}

Using the definition of $G_*(\gamma)$, we know that
\begin{equation*}
\sum_{\left\{\substack{\gamma \subset B(R), \\ \gamma \mbox{ has } k \mbox{ tangencies}}\right\}} \vol(G_*(\gamma)) = \sum_{\left\{ \substack{\gamma \subset B(R), \\ \gamma \mbox{ has } k \mbox{ tangencies}}\right\}}\left( \int\limits_{\{u: \gamma \subset B(u,r)\}} \, du \right).
\end{equation*}

Note that if $u \in B(R-r)$ then we have the inclusion $B(u,r) \subset B(R)$. Hence we have
\begin{align*}
&\sum_{\substack{\gamma \subset B(R), \\ \gamma \mbox{ has } k \mbox{ tangencies}}} \vol(G_*(\gamma)) \geq \\
&\int\limits_{B(R-r)} \left( \sum_{\substack{ \gamma \subset B(u,r)\\ \gamma \mbox{ has } k \mbox{ tangencies} }} \right) \, du = \int\limits_{B(R-r)} \nod_{\zeta}(\Gamma, k, u, r) \, du.
\end{align*}
A division by $\vol(B(r))$ gives us the desired lower bound. A somewhat similar argument gives the proposed upper bound but we leave this to the reader.
\end{proof}

\subsection{Proof of Theorem \ref{thm:euclid_count} (1)-(2): existence of an asymptotic law}

\begin{proof}

We start with $r>0$ fixed and $k \ge 0$ given.
Denote $\nod_{\zeta}(F, k, u, r)$ to be the number of nodal components of $F$ contained in $B(u,r)$ with $k$ $\zeta$-tangencies,
and $\widetilde{\nod}(F,u, r)$ the (total) number of nodal components of $F$ intersecting $\partial B(u,r)$ (cf. Definition \ref{def:counts}).
An application of Lemma \ref{l:integro_geom_sandwich} with $\Gamma=F^{-1}(0)$ and $R$ sufficiently large so that $(1\pm r/R)^{2}$ is $\epsilon$-close to $1$, and bearing in mind the obvious estimate $$\nod^{*}_{\zeta}(F, k, u, r)\le \nod_{\zeta}(F, k, u, r) + \widetilde{\nod}(F,u, r),$$
yields the inequality
\begin{equation}
\label{eq:IGS apply}
\begin{split}
&(1-\epsilon)\cdot\frac{1}{\vol(B(R-r))}
\int\limits_{B(R-r)} \frac{\nod_{\zeta}(F, k, u, r)}{\vol(B(r))} \, du \leq  \frac{\nod_{\zeta}(F, k, R)}{\vol(B(R))} \\
&\le  (1+\epsilon)\cdot\frac{1}{\vol( B(R+r))} \int\limits_{B(R+r)} \frac{\nod_{\zeta}(F, k, u, r) + \widetilde{\nod}(F,u, r)}{\vol(B(r))}.
\end{split}
\end{equation}
Note also that, under the notation of the ergodic Theorem \ref{thm:weiner}, we have the identity
$$\nod_{\zeta}(F, k, u, r)=\nod_{\zeta}(\tau_{u}F, k,r)$$ and $\widetilde{\nod}(F,u, r) = \widetilde{\nod}(\tau_{u}F,r)$,
so that \eqref{eq:IGS apply} reads
\begin{equation}
\label{eq:IGS apply trans}
\begin{split}
&(1-\epsilon)\cdot \int\limits_{B(R-r)} \frac{\nod_{\zeta}(\tau_{u} F, k, r)}{\vol(B(r))} \, du \leq  \frac{\nod_{\zeta}(F, k, R)}{\vol(B(R))} \\
& \le (1+\epsilon)\cdot\frac{1}{\vol( B(R+r))} \int\limits_{B(R+r)} \frac{\nod_{\zeta}(\tau_{u} F, k, r) + \widetilde{\nod}(\tau_{u}F, r)}{\vol(B(r))}.
\end{split}
\end{equation}

An application of Theorem \ref{thm:weiner} on the random variable $$\Phi_{r;\zeta,k}(F):=\frac{\nod_{\zeta}(F, k, r)}{\vol(B(R))},$$ of finite expectation in light of Lemma \ref{lem:Kac Rice Euclid} (1), yields that both
\begin{equation}
\label{eq:int mean -> L1 Crzk}
\frac{1}{\vol(B(R-r))}\int\limits_{B(R-r)} \frac{\nod_{\zeta}(\tau_{u} F, k, r)}{\vol(B(r))} ,\,
\frac{1}{\vol( B(R+r))} \int\limits_{B(R+r)} \frac{\nod_{\zeta}(\tau_{u} F, k, u, r)}{\vol(B(r))} \xrightarrow[L^{1}]{} \widetilde{C}_{r;\zeta,k}
\end{equation}
converge {\em in mean} to $$\widetilde{C}_{r;\zeta,k}=\widetilde{C}_{F,r;\zeta,k}:= \frac{\E[\nod_{\zeta}(\tau_{u} F, k, u, r)]}{\vol(B(r))} \ge 0 .$$
Same argument, now employing
Lemma \ref{lem:Kac Rice Euclid} (2) yields that
\begin{equation}
\label{eq:intrsct->a}
\frac{1}{\vol( B(R+r))} \int\limits_{B(R+r)} \frac{\nod(\tau_{u} F, r)}{\vol(B(r))} \xrightarrow[L^{1}]{} a_{F,r},
\end{equation}
with
\begin{equation}
\label{eq:a const=O(1/r)}
a_{F,r} := \frac{\E[\widetilde{\nod}(\tau_{u}F,r)]}{\vol(B(r))} = \Oc_{r\rightarrow\infty}\left(\frac{1}{r} \right).
\end{equation}

Substituting \eqref{eq:a const=O(1/r)} into \eqref{eq:intrsct->a}, and then together with \eqref{eq:int mean -> L1 Crzk} into \eqref{eq:IGS apply trans} implies that
\begin{equation}
\label{eq:lim const Cauchy}
\E\left[\left| \frac{\nod_{\zeta}(F, k, R)}{\vol(B(R))} -  \widetilde{C}_{r;\zeta,k}  \right|  \right]= \Oc\left( \epsilon+\frac{1}{r}  \right).
\end{equation}
Finally, taking $r\rightarrow\infty$ in \eqref{eq:lim const Cauchy} implies the existence of the limit
$$\widetilde{C}_{k,\zeta}:=\lim\limits_{r\rightarrow\infty}\widetilde{C}_{r;\zeta,k}.$$
Now, recalling the Nazarov-Sodin constant $c(\rho)$ satisfying the defining property \eqref{eq:NS const def prop},
yields the convergence \eqref{eq:conv L1 Euclid}
to the constants $C_{k,\zeta}:=\frac{\widetilde{C}_{k,\zeta}}{c(\rho)}$, at the prescribed rate of $\Oc(1/r)$ of convergence.
Finally, if $k=0$ or $k$ is odd, then a.s. $\nod_{\zeta}(F, k, R)=0$ by Lemma \ref{l:at_least_two_tangencies}, Corollary \ref{cor:tangencies even} and the classical Bulinskaya Lemma \cite[Proposition 6.11]{AW}, stronger than merely $C_{k,\zeta}=0$.
This concludes the proof of (1) of Theorem \ref{thm:euclid_count}.

In case $F$ is isotropic, that $C_{k,\zeta}$ are independent of $\zeta$ follows directly
from the invariance of the law of $F$ w.r.t. the rotations. Alternatively, one can note that, in this case,
the constants $C_{r;\zeta,k}$ as in \eqref{eq:int mean -> L1 Crzk}, a priori
dependent on direction $\zeta$ are, in fact, independent of $\zeta$ thanks to the invariance of $F$ w.r.t. the rotations of $\R^{2}$,
and so are the constants $C_{k,\zeta}$, which is (2) of Theorem \ref{thm:euclid_count}.

\end{proof}

\subsection{Proof of Theorem \ref{thm:euclid_count} (3): support of limit direction distribution}

The deterministic intersection results in \S\ref{s:deterministic} play a crucial role in this section. We remind ourselves that for $k$ odd, $C_{k,\zeta}=0$ follows immediately from the stronger property that $\nod_{\zeta}(F, k, R)=0$ almost surely.

\begin{proof}[Proof of Theorem \ref{thm:euclid_count} (3)]

We distinguish between two main cases: when $(\rho 4^{*})$ is satisfied (i.e. the interior of the support of $\rho$ is non-empty), and when $\rho=\sigma_{S^{1}}$.

{\bf Case $(\rho 4^{*})$ is satisfied.}
The proof follows along similar lines to those of \cite[Theorem 6]{BW17} and \cite[Proposition 5.2]{SW18}, namely the ``barrier method" as initiated in \cite{NS09} and later generalized in \cite{So12}.
The three main steps are standard, and we summarize them as follows: 1) find a deterministic function on $B(r_0)$, for some $r_0>0$ sufficiently large, so that it contains at least one component with the desired number of tangencies, 2) show that the probability of selecting a random function with almost the same nodal component is positive by establishing that $\mathcal{H}(\rho) \subset C^m(B(R))$ is dense and then using the structure of the Gaussian measure on our probability space $\Omega$, and 3) assuming the previous two points, show that the expected density of such nodal components is positive via a simple monotonicity argument originating from \cite{NS09}. Here we give only the details that deviate from referenced article at critical junctures and leave the remaining details to the interested reader.

Recall that by the virtue of axiom $(\rho 4^{*})$, the interior of supp $\rho$ is non-empty. This implies, via the same arguments using the existence of approximations to finite sums of Dirac $\delta$-functions to generate the set of all monomials as in \cite[Section 3.4]{BW17} that the Hilbert space $\mathcal{H}(\rho)$ is dense in $C^m(Q)$, for a given compact set $Q \subset B(R)$ and for all $m$.

Let $\zeta$ be a given direction, and $g_k$ be a $C^2(\R^{2})$ function with a regular nodal component $\Gamma$ that contains the origin, has precisely $k$ tangencies to with respect to the constant vector field $\zeta$ with each tangency being quantitatively transverse, and is contained in $B(r_0)$. Moreover, we know by the deterministic analysis in \S\ref{s:deterministic}, we know there exists a $\beta_0>0$ such that each point of tangency is $\beta_0$-transverse for some $\beta_0>0$ and that each other point of $\Gamma$ is not sub-$\beta_0$. By the density of $\mathcal{H}(\rho) \subset C^k(B(R))$, there exists $\mathfrak{g}_k \in \mathcal{H}(\rho)$ such that $\|\mathfrak{g}_k - g_k \|_{C^2(B(R))} < \beta_0/4$.  Furthermore, we find that the neighborhood defined by $\|\mathfrak{g} - \mathfrak{g}_k \|_{C^2(B(R))} < \beta_0/4$ has only random functions $\mathfrak{g}$ with a nodal component encompassing the origin and with exactly $k$ $\zeta$-tangencies.  The positivity of the quantity $\mathbb{P} \left\{ \|\mathfrak{g} - \mathfrak{g}_k \|_{C^2(B(R))} < \beta_0/4 \right\}$ and consequently of $C_{k, \zeta}$ are standard; see \cite{NS15, SW18, BW17}.

{\bf Case $\rho=\sigma_{S^{1}}$.}
Once again, our proof follows similar lines as that of \cite[Theorem 7]{BW17} and \cite[Proposition 5.2]{SW18} in the case of $\alpha =1$ except that now we utilize the stronger $C^k(B(R))$ versions of both Whitney's approximation/extension theorem and the Lax-Malgrange extenstion theorem; see \cite[Theorem 2.6.3]{Horm} and references therein, and respectively \cite[Lemma 7.2]{EPS}, for the exact statements. For our purposes, $k=2$ is sufficient.

Let $\gamma$ be a smooth closed curve in $\R^2$ that has precisely $k$ (even) tangencies, each of which is quantitatively transverse, with respect to $\zeta$: that is, we know there exists a $\beta_0 >0$ such that each of its tangencies is $\beta_0$-transverse and each other point of $\gamma$ is not sub-$\beta_0$. Via an application of the Implicit Function Theorem, we can find a tubular neighborhood $V_{\gamma}$ of $\gamma$ and a smooth function $H_{\gamma}: V_{\gamma} \rightarrow \R$ with $\gamma = H^{-1}(0)$; furthermore, we can assume that $\inf_{u \in V_{\gamma}} \| \nabla H_{\gamma}(u) \| > 0$.  Now, the $C^2(B(R))$ version of Whitney's theorem states that for a given $\epsilon >0$, there exists a real analytic function $G$ such that $\| G - H_{\gamma} \|_{C^2(V_{\gamma})} < \epsilon$.

Consider the range $\beta_0 > \epsilon > 0$.  By Proposition \ref{p:stab_trans}, for $\epsilon$ sufficiently small, we know that $\widetilde{\gamma} := G^{-1}(0) \cap V_{\gamma}$ has the property that $\mbox{dist}(\widetilde{\gamma}, \gamma) < \epsilon$ (this can also be seen from $C^2(B(R))$ version of Thom's Isotopy Theorem; see \cite{AR}).  Moreover we have that $\widetilde{\gamma}$ also has precisely $k$ tangencies with respect to $\zeta$ each of which is $\beta_0 - \epsilon$-transverse. Now, it remains to find a solution of $\Delta + 1$ that captures this analytic curve $\widetilde{\gamma}$ as part of its nodal set.

Repeating the argument as described in the proof of \cite[Theorem 1.1]{CanSar18}, we use that $\widetilde{\gamma}$ separates $\R^2$ into two components, a bounded component $A_{\widetilde{\gamma}}$ and an unbounded component. Let $\lambda^2$ be the first Dirichlet eigenvalue for the domain $A_{\widetilde{\gamma}}$ and let $h_{\lambda}$ be the corresponding eigenfunction: that is, $(\nabla + \lambda^2) h_{\lambda}(u) = 0$ for $x \in \overline{A_{\widetilde{\gamma}}}$ and $h_{\lambda \, | \widetilde{\gamma}} = 0$.

Let $\overline{\lambda A_{\widetilde{\gamma}}} := \{ x \in \R^2: u/\lambda \in A_{\widetilde{\gamma}} \}$ It follows that $h(u) := h_{\lambda}(u/\lambda)$ solves the Dirichlet problem for $(\Delta + 1)$ on $\overline{\lambda A_{\widetilde{\gamma}}}$ and can be extended (thanks to analyticity) to an open set $B_{\widetilde{\gamma}}$ such that $\overline{\lambda A_{\widetilde{\gamma}}} \subset B_{\widetilde{\gamma}}$ as well as solving the following boundary value-type problem:
\begin{equation*}
\begin{cases}
(\Delta + 1)h(u) = 0, & x \in B_{\widetilde{\gamma}} \\
h(u) = 0, & x \in \lambda \widetilde{\gamma}
\end{cases}
\end{equation*}
where $\lambda \widetilde{\gamma} := \{ x \in \R^2: u/\lambda \in \widetilde{\gamma} \}$. It is important to notice that $\lambda \widetilde{\gamma}$ continues to have precisely $k$ $\frac{(\beta_0-\epsilon)}{\lambda^3}$-transverse tangencies with respect to $\zeta$ thanks to the fact that dilation by $\lambda$ commutes with the action generated by the constant vector field $\zeta$ which is translation in the direction $\zeta$. Added: Now, let $\epsilon' < \frac{\beta_0 - \epsilon}{10\lambda^3}$.

Our next-to-penultimate step is an application of a $C^2(B(R))$ version of Lax-Malgrange \cite[Lemma 7.2]{EPS} which states if $B_{\widetilde{\gamma}}^{\complement}$ has no compact components, then there exists a global solution on $\R^2$ to the equation $(\Delta + 1)g(u) = 0$ with the property that $\|g - h \|_{C^2(B_{\widetilde{\gamma}})} < \epsilon'$. Another application of Proposition \ref{p:stab_trans} implies that $g^{-1}(0) \cap V_{\widetilde{\gamma}}$ has precisely $k$ (transverse) tangencies with respect to $\zeta$. The final step involving disc packing is exactly the same as that for disc packing in the case of $\alpha < 1$.

\end{proof}

\subsection{Proof of Theorem \ref{thm:euclid_count} (4): direction distributions not leaking mass}

\label{sec:proof thm euclid no leak}

\begin{proof}

Given a direction $\zeta\in S^{1}$ and $R>0$, denote $\Ncc_{\zeta}(R)$ to be the number of
$\zeta$-tangency points of $F^{-1}(0)$ lying inside $B(R)$ (where the corresponding nodal component might and might not
be contained in $B(R)$).
The Euclidean analogue of \eqref{eq:tot numb tang=sum mean} is the inequality
\begin{equation}
\label{eq:sum k*nod(k)<=Ncc}
\sum\limits_{k=1}^{\infty}k\cdot \nod_{\zeta}(F,k,R) \le \Ncc_{\zeta}(F,R),
\end{equation}
as we have to discount for those tangency points that lye inside $B(R)$, but belong to the nodal components of $F$
merely intersecting $\partial B(R)$. Since all the involved quantities are nonnegative, we may take the expectation
of both sides of \eqref{eq:sum k*nod(k)<=Ncc} and choose the order of summation and expectation as we please (e.g. by Monotone Convergence Theorem), so that, upon invoking Lemma \ref{lem:Kac-Rice Euclid tangencies}, we may write
\begin{equation*}
\sum\limits_{k=1}^{\infty}k\cdot \E[\nod_{\zeta}(F,k,R)] \le \E[\Ncc_{\zeta}(F,R)]=C_{0}\cdot \vol(B(R)),
\end{equation*}
for some $C_{0}>0$ depending on the law of $F$ only.
Equivalently,
\begin{equation}
\label{eq:sum k*E[k]/E<=C0}
\sum\limits_{k=1}^{\infty}k\cdot \frac{\E[\nod_{\zeta}(F,k,R)]}{\E[\nod(F,R)]} \le C_{0}\cdot \vol(B(R))/\E[\nod(F,R)] \le C_{1},
\end{equation}
with some $C_{1}>0$, holding for $R$ sufficiently large, by \eqref{eq:NS const def prop}.

\vspace{2mm}

We interpret the l.h.s. of \eqref{eq:sum k*E[k]/E<=C0} as the average number of $\zeta$-tangent points per nodal component of $F$,
importantly, bounded by the r.h.s. of \eqref{eq:sum k*E[k]/E<=C0} for every $R>0$ sufficiently large.
The sequence
\begin{equation}
\label{eq:aRk def}
a_{R;k}:=\frac{\E[\nod_{\zeta}(F,k,R)]}{\E[\nod(F,R)]}
\end{equation}
satisfies for every $R>0$ the equality
\begin{equation}
\label{eq:sum(aRk)=1}
\sum\limits_{k= 0}^{\infty}a_{R;k}=1,
\end{equation}
and also, for every $k\ge 0$,
\begin{equation}
\label{eq:lim(aRk)=Ckz}
\lim\limits_{R\rightarrow\infty}a_{R;k}=C_{k,\zeta},
\end{equation}
thanks to Theorem \ref{thm:euclid_count} (1), and \eqref{eq:NS const def prop}. We claim that \eqref{eq:sum(aRk)=1}
together with \eqref{eq:lim(aRk)=Ckz} yield
\begin{equation}
\label{eq:sum(Ck=1)}
\sum\limits_{k\ge 0}C_{k,\zeta}=1,
\end{equation}
which is the statement of Theorem \ref{thm:euclid_count} (4). To this end, it is sufficient to prove the {\em tightness}
of $\{a_{R;k}\}$ w.r.t. $k$, i.e. that for every $\epsilon>0$ there exists $K_{0}=K_{0}(\epsilon)$ and $R_{0}=R_{0}(\epsilon)>0$ sufficiently large,
so that
\begin{equation*}
\sum\limits_{k=K_{0}}^{\infty}a_{R;k}<\epsilon
\end{equation*}
holds for all $R>R_{0}$.
However the said tightness condition follows directly from \eqref{eq:sum k*E[k]/E<=C0} (on recalling that $a_{R;k}$ are given by
\eqref{eq:aRk def}).
As mentioned above, this is sufficient for \eqref{eq:sum(Ck=1)}, which, in turn, is the statement of Theorem \ref{thm:euclid_count} (4).

\end{proof}

\section{Riemannian scenario: Local results}

\subsection{Local setting and statement of the main local result}

Let $V \in \mathcal{V}(\M)$, and define
\begin{equation*}
\nod_{V}(f_{T, \alpha}, k, x, r)
\end{equation*}
to be the number of components of $f_{T, \alpha}^{-1}(0)$ completely contained in the geodesic ball $B(x, r)$ whose number of $V$-tangencies (not counting the zeroes of $V$ that might touch $f_{T, \alpha}^{-1}(0)$) is precisely $k$. That is to say, we are counting the number of  components of $f_{T, \alpha}^{-1}(0)$ completely contained in the geodesic ball $B(x, r)$ with precisely $k$ zeroes of $V f_T$ each.

Recalling the scaled random fields \eqref{eq:fxT scaled def} and Definition \ref{d:blown_up_V} of $\tilde{V}_{x,T}$, the analogous local quantity for $x \in \M \setminus \{V=0\}$ is
\begin{equation*}
\nod_{\tilde{V}_{x,T}}(f_{x,T}, k, R),
\end{equation*}
which, by definition, is equal to the number of components of $f_{x,T}^{-1}(0)$ completely contained in the Euclidean ball $B(0, R)$ whose number of $\tilde{V}_{x,T}$-tangencies is precisely $k$. Equivalently, $\nod_{\tilde{V}_{x,T}}(f_{x,T}, k, R)$ is the number of connected components of
$f_{x, T}^{-1}(0)$ completely contained in $B(0, R)$ with precisely $k$ zeroes of $\tilde{V}_{x,T} f_{x,T}$ each.

\begin{theo} \label{t:conv_prob}
Let $f_{T, \alpha}$ be the random band-limited functions \eqref{eq:fT band lim alpha<1} (or \eqref{eq:fT band lim alpha=1}),
and $V \in \mathcal{V}(\M)$. Then for all $x \in \M \setminus \{ V=0 \}$, $k \ge 0$, and $\epsilon>0$, we have
\begin{equation}
\label{eq:double lim R,T conv loc}
\lim\limits_{R \rightarrow \infty} \limsup\limits_{T \rightarrow \infty} \, \mathbb{P} \left[ \left| \frac{\nod_{V}(f_T, k, x, R/T)}{c_{2, \alpha} \cdot \vol(B(R))}  - C_{k} \right| > \epsilon \right] = 0
\end{equation}
where the constants $C_{k}$ are same as in Theorem \ref{thm:euclid_count} (1)
applied on the random field $F=\gfr_{\alpha}$ with $\zeta$ arbitrary
($C_{k}$ are independent of $\zeta$ by Theorem \ref{thm:euclid_count} (2)), and $c_{2,\alpha}>0$ is the Nazarov-Sodin constant
of $\gfr_{\alpha}$, i.e. satisfying the defining property \eqref{eq:NS const def prop} with $F=\gfr_{\alpha}$.
\end{theo}

\subsection{Exceptional events} \label{s:exceptional_events}

In this section, we introduce the following parameters and their purposes: a small parameter $\delta>0$ to control probabilities, $b>0$ to control the quality of coupling approximation, $\beta_{1,i} > 0$ to control the regularity of our components of $\mathfrak{g}_{\alpha}$ for $i=1,2$, $\beta_{2,i} >0$ to control the nature of our tangencies for $i=1,2$, $M_i$ to control the $C^k$ norms of our random functions for $i=0,1$, large spectral parameter $T>0$, and related $R>0$ controlling our convergence to the scaling limit.

\begin{defn}[Exceptional events] \label{d:standard_events}
Let $R,T, b, M_0, M_1, \beta_1, \beta_2 >0$ and $x \in \M \setminus \{ V=0 \}$. We define the events as follows:
\begin{align*}
& \Delta_1 = \Delta_1(x, R,T;b) = \{ \|f_{x,T} - \mathfrak{g}_{\alpha} \|_{C^2(B(2R))} \geq b \} \\
& \Delta_2 = \Delta_2(R, M_0) = \{ \| \mathfrak{g}_{\alpha} \|_{C^2(B(2R))} \geq M_0 \} \\
& \Delta_3 = \Delta_3(x, R, T, M_1) = \{ \| f_{x,T} \|_{C^2(B(2R))} \geq M_1 \} \\
& \Delta_4 = \Delta_4(R, \beta_{1,1}) = \{ \min_{u \in B(2R)} \max \{ |\mathfrak{g}_{\alpha}(u)|, \|\nabla \mathfrak{g}_{\alpha}(u)\|_{2}  \} \leq \beta_{1,1} \} \\
& \Delta_5 = \Delta_5(x, R, T, \beta_{1,2}) = \{ \min_{u \in B(2R)} \max \{ |\tilde{V}_{x,T}(\mathfrak{g}_{\alpha})(u)|, \|\nabla (\tilde{V}_{x,T}\mathfrak{g}_{\alpha})(u)\|_{2}  \} \leq \beta_{1,2} \}.
\end{align*}
\end{defn}

We would like to emphasize that for the event $\Delta_1$, its analogue in \cite[Lemma 4]{So12}  only involves the $C^1$ norm and is stated to have a quantitatively low probability. It is however clear, thanks to the flexibility of the Hadamard-Landau inequality, that we can immediately replace $C^1$ by $C^2$ both in the statement of \cite[Lemma 4]{So12} and its proof. The parameter $\beta_{1,2}$ in some sense measures our distance away from the set $\{ V = 0\}$ on $\M$.

\begin{defn} \label{d:unif_stability_fields}
 For $R>0$, and $\eta, \beta_{1,3}>0$ small, we define the \textit{unstable components} event $\Delta_6(\eta, R, \beta_{1,3})$ as follows:
\begin{align*}
\Delta_6(\eta, R,\beta_{1,3}) := &\{ \mbox{the number of components of $\mathfrak{g}_{\alpha}^{-1}(0)$ with a point $u_0$} \\
& \mbox{such that } | \nabla \mathfrak{g}_{\alpha}(u_0)| \leq  \beta_{1,3} \mbox{ is at least $\eta R^2$} \}
\end{align*}
\end{defn}

\begin{defn}  \label{d:unstable_event}
For $x \in \M\setminus \{V = 0\}$, $R, T>0$, and $\eta, \beta_2>0$ small, we define the \textit{tangentially unstable components} event $\Delta_7(V, x, R, T, \beta_2, \eta)$ as follows:
\begin{align*}
\Delta_7(V, x, R, T, \beta_2, \eta) := &\{ \mbox{the number of components of $\mathfrak{g}_{\alpha}^{-1}(0)$, with point $u_0$ such that} \\
& \det [\nabla \mathfrak{g}_{\alpha}(u_0),  \nabla \left( \tilde{V}_{x,T} \, \mathfrak{g}_{\alpha} \right) (u_0) ] | \leq  \beta_2 \mbox{, is at least $\eta R^2$} \}
\end{align*}
where $\tilde{V}_{x,T}$ is the blown-up vector field at $x$, at scale $T$, as given in Definition \ref{d:blown_up_V}.
\end{defn}

A combination of the Borel-TIS and Sudakov-Fernique inequalities along with the argument in \cite[Section 4.3]{BW17} using Chebyshev's Inequality give us a probability of less than $\delta$ for the events $\Delta_2$ and $\Delta_3$.

\begin{lem}[Uniform stability for components of smooth Gaussian fields, Cf. {\cite[Proposition 4.3]{SW18}}] \label{p:unstable_comp}
Given $\delta>0$, $\eta >0$, there exists $\beta_{1,3}(\delta,\eta)>0$ (possibly small) and $R_0(\delta, \eta)$ such that for all $R \geq R_0$, we have that $\mathbb{P}[\Delta_6(\eta, R, \beta_{1,3})] < \delta$.
\end{lem}

\begin{prop}[Uniform stability of $\beta_2$-transverse tangencies for smooth Gaussian fields] \label{p:unstable_tangencies}
Given $\eta, \delta > 0$ and $x \in \M\setminus\{V = 0\}$, there exists $\beta_{2}=\beta_{2}(\eta,\delta, V)>0$, and in turn
$R_0(\delta, \eta, \beta_{2},x)$ such that for all $R \geq R_0$, there exists a $T_0(R, \beta_{2}, V)>0$ such that for all $T \geq T_0$, we have $\mathbb{P}[\Delta_7(V,x,R, T, \beta_2, \eta)] < \delta$.
\end{prop}

We would like to remind the reader of the following standard fact:
\begin{lem} \label{l:qual_bulinskaya}
For all $x \in \M \setminus \{V=0\}$, we have that the event
\begin{align*}
\Delta_8(V, x, T, R) &= \bigg\{ \mbox{there exists } u_0 \in B(R) \mbox{ such that } \mathfrak{g}_{\alpha}(u_0)=\left(\tilde{V}_{x,T} \, \mathfrak{g}_{\alpha} \right) (u_0) \\
&=  \det \left[\nabla \mathfrak{g}_{\alpha}(u_0),  \nabla \left( \tilde{V}_{x,T} \, \mathfrak{g}_{\alpha} \right)(u_0) \right] = 0 \bigg\}
\end{align*}
has probability 0.
\end{lem}

\begin{proof}[Proof of Lemma \ref{l:qual_bulinskaya}]
This is an immediate application of Bulinskaya's lemma \cite[Proposition 6.11]{AW} on the random field $(\mathfrak{g}_{\alpha}, \tilde{V}_{x,T} \mathfrak{g}_{\alpha})$.
\end{proof}

Thanks to having both $\mathfrak{g}_{\alpha}^{-1}(0)$ and $( \tilde{V}_{x,T} \mathfrak{g}_{\alpha})^{-1}(0)$ (and the zero sets for the corresponding perturbations) being regular with probability 1, that the intersection between these two sets is always transversal by Lemma \ref{l:qual_bulinskaya}, and finally Lemma \ref{l:even_int}, we are able to focus only on events where the nodal components have an even number of tangencies to $\tilde{V}_{x,T}$ (although $k=0$ may occur and is addressed in the proof of Proposition \ref{p:stab_k_even}). Thus, we have reduce our analysis to that of $\beta$-transverse tangencies and sub-$\beta$ tangencies.

In closing this, we set
\begin{equation}
\label{eq:E good event def}
E = \cap_{i=1}^8 \Delta_i^{\complement}
\end{equation}
with the corresponding range of parameters so that $\mathbb{P}[E] > 1 - \delta$, where $\delta > 0$ is given. We will refer to this $E$ numerous times in upcoming sections.

\subsection{High probability stability estimates for local counts} \label{s:stab_est}

A $\mbox{sub}-\frac{3}{2} \beta$ tangency of $\mathfrak{g}_{\alpha}^{-1}(0)$ is mapped, under a small $C^{2}$ perturbation, to $\mbox{sub}-\left( \frac{3}{2}\beta + b \right)$ tangency of $f_{x,T}^{-1}(0)$, where $b$ is small (see Definition \ref{d:standard_events}).  We will later prove that components containing these near-degenerate tangencies are {\em few} in number, with high probability (see Proposition \ref{p:unstable_tangencies}).


\begin{prop}[Stability for the number of tangencies] \label{p:stab_k_even}
Let $V \in \mathcal{V}(\M)$ be given, and let $\tilde{V}_{x,T}$ denote the blow-up of $V$ at $x$. We have that for all $x \in \M \setminus \{V = 0\}$, $\delta>0$, $\eta >0$, and $k \ge 0$, there exists a $R_0(x, \delta, \eta, \|V(x)\|)$ such that for all $R \geq R_0$, there exists a $T_0(R, \|V(x)\|, \delta, k)$ such that for all $T \geq T_0$, we have the estimate
\begin{align}
\label{eq:perturb count}
 \nod_{\tilde{V}_{x,T}}(\mathfrak{g}_{\alpha}, k , R-1) - \eta R^2 & \leq \nod_{\tilde{V}_{x,T}}(f_{x,T}, k , R)
 \leq \nod_{\tilde{V}_{x,T}}(\mathfrak{g}_{\alpha}, k , R+1) +  \eta R^2,
\end{align}
with probability $> 1 - \delta.$
\end{prop}

Note that in Proposition \ref{p:stab_k_even} we are counting all $k$ tangencies, whether they are $\beta$-transversal or sub-$\beta$. This is the source of the extra terms of $\pm\eta R^{2}$ on both l.h.s. and r.h.s. of \eqref{eq:perturb count}, accounting for the sub-$\beta$, or rather near-degenerate, tangencies.


Let $V \in \mathcal{V}$. We will apply our deterministic intersection results from \S\ref{s:deterministic} to smooth Gaussian fields $G = (\mathfrak{g}_{\alpha}, \tilde{V}_{x,T}\mathfrak{g}_{\alpha})$, where $\tilde{V}_{x,T}$ is the blow-up of $V$ at $x \in \M$,  and $(f_{x,T}, \frac{1}{T}(Vf_T)_{x,T})$. It is important to note that there two advantages to considering $\frac{1}{T}(Vf_T)_{x,T}$. First $(\frac{1}{T} V)f_T$ and $ V(f_T)$ share the same zero sets and second, the zeroes of $\frac{1}{T}(Vf_T)_{x,T}$ and $\tilde{V}_{x,T}(f_{x,T})$ are close with respect to a natural perturbation parameter that appears when summoning certain coupling results (cf. \cite{SW18, So12} and Definition \ref{d:standard_events}).

\begin{lem}[Blowing up ``almost" commutes with differentiation] \label{l:non_commuting}
Given $x \in \M\setminus\{V=0\}$, and $R>0$, there exists $T_0(x,R,b)$ such that for all $T \geq T_0$, $\tilde{V}_{x,T}$ is non-vanishing in $B(R) \subset T_x\M$ and we have that
\begin{align*}
 T \tilde{V}_{x,T} \left( f_{x,T} \right)  = (Vf_T)_{x,T}+ \mathcal{O}_{\|f_{x,T}\|_{C^1}}\left( \frac{R}{T}\right).
\end{align*}
\end{lem}
\begin{proof}
This is an immediate consequence of Taylor's theorem and the fact that for $T$ large enough, $R/T < \min \{inj(\M), 1 \}$ and that $D(\exp_x)_{Z = \frac{Y}{T}} = Id + \mathcal{O} \left( \frac{R}{T} \right)$ as $|Y|/T \leq R/T$.
\end{proof}

We need an intermediate estimate on the plane before proceeding to the proof of Theorem \ref{t:conv_prob}:
\begin{lem}[Close vector fields generate the same count] \label{l:close_vector_fields}
Let $\delta > 0$ be given.  Consider the parameter $\beta_2$ that controls degeneracy of our tangencies as described in \S\ref{s:exceptional_events}.  Let $0 < \epsilon \leq \beta_2$ be given. Suppose $\tilde{V}_1, \tilde{V}_2 \in \mathcal{V}(\R^2)$ are non-vanishing and that $\| \tilde{V}_1 - \tilde{V}_2 \|_{C^1(B(R))} < \epsilon$; that is, their coefficients with respect to the standard basis are $C^1$ close.
Then, for all $k \ge 0$, for all $\eta >0$, on the event $E$ as in \eqref{eq:E good event def} (of probability $>1-\delta$), there exists $R_0(\eta, \epsilon)$ such that for all $R \geq R_0$, the following inequality holds:
\begin{align*}
\nod_{\tilde{V}_1}(\mathfrak{g}_{\alpha}, k , R-1) - \eta R^2 & \leq \nod_{\tilde{V}_2}(\mathfrak{g}_{\alpha}, k , R) \leq  \nod_{\tilde{V}_1}(\mathfrak{g}_{\alpha}, k , R+1) + \eta R^2.
\end{align*}
\end{lem}

\begin{proof}[Proof of Lemma \ref{l:close_vector_fields}]
The proof follows almost every step of the proof of our main Proposition \ref{p:stab_k_even} given in \S\ref{sect:proof_main_stab_prop} except that we do not perturb by the vector function $( f_{x,T} - \mathfrak{g}_{\alpha}, \tilde{V}_{x,T}(f_{x, T}  - \mathfrak{g}_{\alpha}) )$ but by the simpler function $(0, (V_2-V_1) \mathfrak{g}_{\alpha})$. This type of perturbation is under the scope of the hypotheses of Lemma \ref{l:Sodin_type_calc_lem} where we set $b < \epsilon$ and $\Psi = (0, (V_2-V_1) \mathfrak{g}_{\alpha})$.
\end{proof}

\begin{proof}[Proof of Theorem \ref{t:conv_prob} assuming Proposition \ref{p:stab_k_even}, Lemma \ref{l:non_commuting} and Lemma \ref{l:close_vector_fields}]

Considering the listed proposition and lemmas, we let $\delta, \eta>0$ be given and $k\ge 0$. Now take the parameters as prescribed in \S\ref{s:exceptional_events}: that is, take $R_0, M_0, M_1, \beta_{1,1}, \beta_{1,2}, \beta_{1,3}, \beta_2$, and $T_0$.

We assume that
\begin{equation} \label{e:final_freq_stab}
R/T < \min \left\{1, \mbox{inj}(\M), \frac{1}{10 M_0}, \beta_{1,1}, \beta_{1,2}, \beta_{1,3}, \beta_2 \right\}
\end{equation}
after possibly taking the initial $T_0$ larger than is posed in \S\ref{s:exceptional_events}.  We know that $(Vf_T)_{x,T} = T (\tilde{V}_{x,T})(f_{x,T}) + \mathcal{O}_{\|f_{x,T}\|_{C^1}}\left( \frac{R}{T} \right)$ in the local coordinates w.r.t. $\exp_x$ by Lemma \ref{l:non_commuting}.  Let us define a new quantity $\nod \left(f_{x,T}, (Vf_T)_{x,T}, k,  R  \right)$ and set it equal to the number of components of $f_{x,T}^{-1}(0)$ contained inside of $B(R)$ whose number of points of intersection with $(Vf_T)_{x,T}^{-1}(0)$ is exactly $k$.  Recall that in Section 4, we established that understanding the number of tangencies was equivalent to understanding the number of intersections between two sets of curves.  From this, we see immediately that $\nod_{V}(f_T, x, k, \frac{R}{T})$ is bounded above and below by
\begin{equation*}
\nod \left(f_{x,T}, (Vf_T)_{x,T}, k,  R \pm 1  \right).
\end{equation*}

Now by our choice of $T_0$ as indicated after \eqref{e:final_freq_stab}, Lemma \ref{l:close_vector_fields}
(in particular, its proof which uses the equivalence between counting tangencies and counting intersections), and Corollary \ref{l:Sodin_annuli}, we have that $\nod \left(f_{x,T}, (Vf_T)_{x,T}, k,  R \pm 1  \right)$ itself is bounded above and below
by $\nod_{T\tilde{V}_{x,T}}\left( f_{x,T}, k,  R \pm 2  \right) \pm \eta R^2.$ Thus,
\begin{equation*}
 \nod_{T \, \tilde{V}_{x,T}}(f_{x,T}, k, R-2)  - \eta R^2 \leq \nod_{V}\left( f_T,x, k,  \frac{R}{T}  \right) \leq \nod_{T \, \tilde{V}_{x,T}}(f_{x,T}, k, R+2)   + \eta R^2.
\end{equation*}
Now, we emphasis that $\nod_{T \, \tilde{V}_{x,T}}(f_{x,T}, k, R \pm 2) =  \nod_{ \tilde{V}_{x,T}}(f_{x,T}, k, R \pm 2)$ since joint zeroes of $(f_{x,T}, T \, \tilde{V}_{x,T}(f_{x,T}))$ are the same as those of $(f_{x,T}, \tilde{V}_{x,T}(f_{x,T}))$, therefore giving us a one-to-one correspondence between components of $f_{x,T}^{-1}(0)$ with precisely $k$ zeroes of $\tilde{V}_{x,T}(f_{x,T})$ and those with precisely $k$ zeroes of $T\tilde{V}_{x,T}(f_{x,T})$.  This leads us to
\begin{equation*}
 \nod_{\tilde{V}_{x,T}}(f_{x,T}, k, R-2)  - \eta R^2 \leq \nod_{V}\left( f_T,x, k,  \frac{R}{T}  \right) \leq \nod_{ \tilde{V}_{x,T}}(f_{x,T}, k, R+2)   + \eta R^2.
\end{equation*}

In the local coordinates around $x$ given by $y=\exp_x(Y)$, we set $\tilde{V}(x)$ to be the constant vector field given by the trivial extension of $(\tilde{V}_{x,T})_{|Y=0}$ to $\R^2$. Proposition \ref{p:stab_k_even} yields us with probability $1 - \delta$ that there exists $R_1$ such that $R \geq R_1(x)$ and $T_1(R)$ such that $T \geq T_1$, along with an application of the vector field comparison statement Lemma \ref{l:close_vector_fields} in comparing $\tilde{V}_{x,T}$ to the constant field $\tilde{V}(x)$ by taking $T$ as in (\ref{e:final_freq_stab}), we have
\begin{align*}
 \nod_{\tilde{V}_{x,T}}(\mathfrak{g}_{\alpha}, k , R-1) - 2\eta R^2 & \leq \nod_{\tilde{V}(x)}(f_{x,T},  k , R)  \leq \nod_{\tilde{V}_{x,T}}(\mathfrak{g}_{\alpha}, k , R+1) + 2\eta R^2.
\end{align*}
Now, apply Lemma \ref{l:close_vector_fields} for the second time in the case of $f_{x,T}$, $V_1 = \tilde{V}_{x,T}$, $V_2 = \tilde{V}(x)$, and the same values of $R,T$ as in \eqref{e:final_freq_stab}, to obtain
\begin{equation*}
 \nod_{\tilde{V}(x)}(f_{x,T}, k, R-3)  - \eta R^2 \leq \nod_{V}\left( f_T,x, k,  \frac{R}{T}  \right) \leq \nod_{\tilde{V}(x)}(f_{x,T}, k, R+3)   + \eta R^2.
\end{equation*}

Let $\epsilon_1>0$ be given. By the above estimates, occurring with probability $> 1 - \delta$, this implies
\begin{align*}
\mathbb{P}\left[ \left| \frac{ \nod_{V}\left( f_T, x,  k , \frac{R}{T} \right)}{c_{2, \alpha} \pi R^2} - C_k \right| > \epsilon_1 \right] & \leq
\mathbb{P}\left[ \left| \frac{ \nod_{\tilde{V}(x)}(\mathfrak{g}_{\alpha}, k , R-4) }{c_{2, \alpha} \pi R^2} - C_k \right| > \epsilon_1 - \frac{4\eta}{c_{2, \alpha} \pi} \right] + \\
& \mathbb{P}\left[ \left| \frac{ \nod_{\tilde{V}(x)}(\mathfrak{g}_{\alpha}, k , R+4) }{c_{2, \alpha} \pi R^2} - C_k \right| > \epsilon_1 - \frac{4\eta}{c_{2, \alpha} \pi} \right] +  \mathcal{O}(\delta); \\
\end{align*}
note we have used that $R^2 = (R - 2)^2 + \mathcal{O}(R)$ in $R$ along with the Kac-Rice formula applied to critical points for $\mathfrak{g}_{\alpha}$ as in \cite[Corollary 2.3]{SW18} to show that
\begin{equation*}
\frac{ \nod_{\frac{\partial}{\partial Y_1}}(\mathfrak{g}_{\alpha}, k , R-4) }{c_{2, \alpha} \pi R^2}  = \frac{ \nod_{\frac{\partial}{\partial Y_1}}(\mathfrak{g}_{\alpha}, k , R-4) }{c_{2, \alpha} \pi (R-4)^2} + o(1)
\end{equation*}
for $R \geq R_3$.  To guarantee that $ \epsilon_1 - \frac{4\eta}{c_{2, \alpha} \pi} > \frac{\epsilon_1}{2}$, we choose $\eta$ possibly smaller.

Thanks to all of this, we have that for any radius parameter $R \geq R_4$, we can take $T \geq T_4$ (for $R_4$ and $T_4$ chosen possibly larger once again after invoking our Euclidean nodal count result Theorem \ref{thm:euclid_count}) to obtain the estimate
\begin{equation*}
\mathbb{P}\left[ \left| \frac{ \nod_{\tilde{V}(x)}(\mathfrak{g}_{\alpha}, k , R \pm 4) }{c_{2, \alpha} \pi (R \pm 4)^2} - C_k \right| > \frac{\epsilon_1}{2} \right] = o(1)
\end{equation*}
which in turn implies that
\begin{equation*}
\mathbb{P}\left[ \left| \frac{ \nod_{V}\left( f_T, x,  k , \frac{R}{T} \right)}{c_{2, \alpha} \pi R^2} - C_k \right| > \epsilon_1 \right] = o(1)
\end{equation*}
This concludes our proof.
\end{proof}

\subsection{Proof of Proposition \ref{p:stab_k_even}: stability estimate} \label{sect:proof_main_stab_prop}

In this section we always take $\tilde{V} \in \mathcal{V}(\R^2)$ to be nowhere vanishing.

\begin{defn}

\begin{enumerate}

\item We denote by $\nod_{\tilde{V}}^{\beta-\mbox{trans}}(\mathfrak{g}_{\alpha}, k, R)$ the number of connected components of
$\mathfrak{g}_{\alpha}^{-1}(0)$, strictly contained in $B(R)$, whose number of $\tilde{V}$-tangencies is precisely $k$ with each point of tangency being $\beta$-transverse (cf. Definition \ref{d:beta_trans}).

\item We denote by $\nod_{\tilde{V}}^{\mbox{sub}-\beta}(\mathfrak{g}_{\alpha}, k, R)$ the number of connected components of
$\mathfrak{g}_{\alpha}^{-1}(0)$, strictly contained in $B(R)$,  whose number of $\tilde{V}$-tangencies is precisely $k$ with at least one point of tangency being sub-$\beta$ (cf. Definition \ref{d:sub_beta}).

\end{enumerate}

\end{defn}

\begin{prop}[Stability of $\beta$-transversality] \label{p:b_trans_est}
Let $k \ge 0, \eta, \delta > 0$, and $x \in \M \setminus \{V = 0\}$ be given. Then there exists $\beta_0(x)$ and $R_0(V,x, \eta)>0$ such that for all $\beta < \beta_0$ and $R \geq R_0$, there exists $T_0(V,x, R, \beta, \eta)>0$ and $b(\beta, V)>0$ such that for all $T \geq T_0$, we have the estimate
\begin{align*}
\nod_{\tilde{V}_{x,T}}^{\frac{3}{2}\beta-\mbox{trans}}(\mathfrak{g}_{\alpha}, k, R-1) & \leq \nod_{\tilde{V}_{x,T}}^{\beta-\mbox{trans}}(f_{x,T}, k, R)
 \leq \nod_{\tilde{V}_{x,T}}^{\frac{1}{2}\beta-\mbox{trans}}(\mathfrak{g}_{\alpha}, k, R+1)
\end{align*}
for an event $E$, where $\mathbb{P} \left[ E \right] > 1 - \delta$ and $\|f_{x,T} - \mathfrak{g}_{\alpha} \|_{C^2(B(R))} < b$.
\end{prop}

\begin{proof}[Proof of Proposition \ref{p:stab_k_even} assuming Proposition \ref{p:b_trans_est} and Proposition \ref{p:unstable_tangencies}]
We prove the first inequality as the second inequality will be proved in an identical fashion. Let $\eta>0$ be given as in the referenced propositions, and assume that $E$ as in \eqref{eq:E good event def} occurs.
Since we have
\begin{align*}
\nod_{\tilde{V}_{x,T}}(\mathfrak{g}_{\alpha}, k, R-1) &= \nod_{\tilde{V}_{x,T}}^{\frac{3}{2}\beta-\mbox{trans}}(\mathfrak{g}_{\alpha}, k, R-1) + \nod_{\tilde{V}_{x,T}}^{\mbox{sub}-\frac{3}{2}\beta}(\mathfrak{g}_{\alpha}, k, R-1),
\end{align*}
Proposition \ref{p:unstable_tangencies} implies that for $\beta$ possibly smaller than $\beta_0(\eta) = \beta_2$ as initally given in \S\ref{s:exceptional_events}, that
\begin{align*}
\nod_{\tilde{V}_{x,T}}(\mathfrak{g}_{\alpha}, k, R-1) \leq \nod_{\tilde{V}_{x,T}}^{\frac{3}{2}\beta-\mbox{trans}}(\mathfrak{g}_{\alpha}, k, R-1) + \eta R^2
\end{align*}
for all $R \geq R_0$ large enough and $T \geq T_0$ with $T_0$ being large enough and depending on $R$ amongst other parameters.  We are now in a position to apply Proposition \ref{p:b_trans_est} and repeat this argument for $\nod_{\tilde{V}_{x,T}}(f_{x,T}, k, R)$; that is, we write
\begin{align*}
& \nod_{\tilde{V}_{x,T}}(\mathfrak{g}_{\alpha}, k, R-1) \leq \nod_{\tilde{V}_{x,T}}^{\beta-\mbox{trans}}(f_{x,T}, k, R) + \eta R^2 \\
& = \left( \nod_{\tilde{V}_{x,T}}(f_{x,T}, k, R) - \nod_{\tilde{V}_{x,T}}^{\mbox{sub}-\beta}(f_{x,T}, k, R) \right) + \eta R^2 \\
& \leq  \nod_{\tilde{V}_{x,T}}(f_{x,T}, k, R) + \eta R^2
\end{align*}
since $\nod_{\tilde{V}_{x,T}}^{\mbox{sub}-\beta}(f_{x,T}, k, R) \geq 0$. Note that for the upper bound of $ \nod_{\tilde{V}_{x,T}}(f_{x,T}, k, R)$, a similar argument follows.
\end{proof}

\begin{proof}[Proof of Proposition \ref{p:b_trans_est} assuming the results in \S\ref{s:exceptional_events} and \S\ref{s:local_results}]

Let $\eta, \delta>0$ be given and $\mathfrak{G} = (\mathfrak{g}_{\alpha}, \tilde{V}_{x,T}\mathfrak{g}_{\alpha})$. Now, consider the multi-parameter-dependent event $$E(x,R,T, \eta, \delta, \beta_{1,1}, \beta_{1,2}, \beta_{1,3}, \beta_2, b),$$ as in \eqref{eq:E good event def}.
Let $\beta_0 := \min \{  \beta_{1,1}, \beta_{1,2}, \beta_{1,3}, \beta_2\}$ and set $b < \frac{\beta_0}{4}$. We now satisfy the hypotheses of Proposition \ref{p:stab_trans}.  On the corresponding event $E$, we obtain some mapping properties of the components of $\mathfrak{g}_{\alpha}^{-1}(0)$ to the components of $f_{x,T}$ after perturbations of $C^1$ size $b$. More specifically, we know that after adding the perturbation $f_{x,T} - \mathfrak{g}_{\alpha}$ to $\mathfrak{g}_{\alpha}$, we have that each of the $\frac{3}{2} \beta_2$-transverse tangencies of $\mathfrak{g}_{\alpha}^{-1}(0)$ goes to a $ \left( \frac{3}{2}\beta_2 - b \right)$-transverse tangency of $f_{x,T}^{-1}(0)$. Thus, by Corollary \ref{l:Sodin_annuli} a component of $\mathfrak{g}_{\alpha}^{-1}(0)$ whose number of tangencies is precisely $k$ and each of which is $\frac{3}{2} \beta_2$-transverse is mapped to a component of $f_{x,T}^{-1}(0)$ whose number of tangencies is precisely $k$ and each of which is $(\frac{3}{2} \beta_2 - b)$- transverse. Thus, we have some form of injectivity and it follows that
\begin{equation*}
\nod_{\tilde{V}_{x,T}}^{(\frac{3}{2}\beta_2 )-\mbox{trans}}(\mathfrak{g}_{\alpha}, k, R-1) \leq \nod_{\tilde{V}_{x,T}}^{(\frac{3}{2}\beta_2 - b)-\mbox{trans}}(f_{x,T}, k, R).
\end{equation*}
The second inequality follows similarly.
\end{proof}

\subsection{Proof of Proposition \ref{p:unstable_tangencies}: high probability of few components with near-degenerate tangencies}
In this section, let $T>0$ and $\mathfrak{G}: \R^2 \rightarrow \R^2$ be the smooth Gaussian field
\begin{equation}
\label{eq:frakG def}
\mathfrak{G}(u) := \left( \mathfrak{g}_{\alpha}(u), \tilde{V}_{x,T}\mathfrak{g}_{\alpha}(u)\right).
\end{equation}
It is important to note that all the results in this section also hold for the Gaussian field $\mathfrak{F}(u) := \left( f_{x,T}(u), \tilde{V}_{x,T}f_{x,T}(u)\right)$ thanks to the exclusion of the $C^2$ decoupling event $\Delta_1$ as in \S\ref{s:exceptional_events}.

\begin{defn}
Let $\eta>0$, and $\beta_{2,2}>0$ be small, $R, R_1>0$ be large parameters such that $0< R_1 < R$ and $R/R_1$ is itself large. Let us cover $B(R)$ with approximately $(R/R_1)^2$ balls $\mathcal{D}_i$ of radius $R_1$ such that the multiplicity of the covering is bounded by a constant $\kappa>0$.  Denote by $\mathcal{G}_i$ the balls centred at the same points as $\mathcal{D}_i$ with radii $3R_1$. Note that the multiplicity of the covering $\{ \mathcal{G}_i \}_i$ is bounded by $c_0(2) \kappa$.

\begin{enumerate}
\item We say that the smooth random field $G=(g_1,g_2)$ is $( \beta_{2,2}, 3R_1)$-stable on a ball $\mathcal{G}_i$ if for all $u \in \mathcal{G}_i$, we
have $|g_1(u)| > \beta_{2,2}$ or
$|\det DG(u)| > \beta_{2,2}$; otherwise we say that
$G$ is $(\beta_{2,2})$-unstable on $\mathcal{G}_i$.
\item We say that $G$ is $(\eta, \beta_{2,2}, 3R_1)$-stable if $G$ is $(\beta_{2,2}, 3R_1)$-stable on all $\mathcal{G}_i$ except
for up to $\eta R^2$ ones.
\end{enumerate}
\end{defn}

\begin{prop} \label{p:abc_stability}
Let $x \in \M \setminus \{V=0\}$, and $R_{1}>0$. Given $\delta>0$, $\eta>0$, there exists a positive number $\beta_{2,2}(\delta,R_{1},\eta)$ so that for all $R>0$ with $R/R_1 > 100$, there exists a spectral parameter value $T_0(V, R, \beta_{2,2})$ such that for all $T \geq T_0$, any random field $\mathfrak{G}_{|B(R)}$ as in \eqref{eq:frakG def} is also $(\eta,\beta_{2,2}, 3R_1)$-stable
with probability $1 - \delta$.
\end{prop}

Note that there is an implicit dependence of all the parameters above on the radius $R_1$ for the elements in our covering. We will track this throughout our calculations. To prove this proposition, we quote the following two lemmas from \cite{SW18}:

\begin{lem}[Cf. {\cite[Lemma 4.9]{SW18}}] \label{l:4_separation}
There exists a constant $c_0 = c_0(\kappa) > 0$ depending only on $\kappa$ with the following property: for $\mathcal{G} = \{ \mathcal{G}_i \}_{i \leq K}$ a collection of radius-$3R_1$ balls lying in $B(R)$ such that each point $x \in B(R)$ lies in at most $\kappa$ elements of $\mathcal{G}$, we have that $\mathcal{G}$ contains at most $c_0 \, K$ balls that are in addition $4$-separated.
\end{lem}

\begin{lem}[Cf. {\cite[Lemma 4.10]{SW18}}] \label{l:derivative_bounds}
Let $F$ be a stationary Gaussian random field. For all $\delta>0$ and for all $K, m \in \mathbb{N}$, there exists $C_0(\epsilon) >0$ such that for any (possibly random) collection of centers $\{ u_i \}_{i \leq K}$ satisfying $d(u_i, u_j) > 4$ for $i \neq j$, there exist $\lfloor K/2 \rfloor$ points $\{ u_{i_j} \}_{i_j \in I}$ with $|I| = \lfloor K/2 \rfloor$ such that
\begin{equation} \label{e:small_derivatives}
\sup\limits_{| \nu | \leq m, B(u_{i_j}, 1)} | \partial^{\nu} F(u) | \leq C_0 \frac{R}{\sqrt{K}}
\end{equation}
holds with probability $> 1 - \delta$.
\end{lem}

For the field $F = \mathfrak{g}_{\alpha}$, the constant $C_0$ in (\ref{e:small_derivatives}) depends on the quantity $\sum_{|\nu| \leq s_0} \E \left[  | \partial^{\nu} F(0) |  \right]$, thanks to our Gaussian process being stationary. For $F= \tilde{V}_{x,T} \mathfrak{g}_{\alpha}$,
Lemma \ref{l:derivative_bounds} is not directly applicable, since $F$ is no longer stationary. However, in what follows, we argue that the statement
\eqref{e:small_derivatives} of Lemma \ref{l:derivative_bounds} does hold with $F$, possibly increasing the $C_{0}$ in \eqref{e:small_derivatives}.
First, the derivatives of $F$
could be expressed in terms of the derivatives of $\gfr_{\alpha}$ and the derivatives of $\tilde{V}_{x,T}$.
Also, the vector field $\tilde{V}_{x,T}$ is asymptotically constant, as can be seen from Taylor expanding the coefficients $(a_1)_{x,T}$ and $(a_2)_{x,T}$ as in \eqref{eq:tild(V) def},
and using $$D(\exp_x)_{Z = \frac{Y}{T}} = Id + \mathcal{O} \left( \frac{R}{T} \right)$$ for $|Y|/T \leq R/T$.
Hence it follows that \eqref{e:small_derivatives} holds, with the constant $C_0$ in the r.h.s. of \eqref{e:small_derivatives} now depending on the derivatives of these coefficients of $V$, in Riemannian normal coordinates, evaluated at $0$.

\vspace{2mm}

In preparation for the proof of Proposition \ref{p:abc_stability}, we need an estimate on some volume quantities.

\begin{lem}[Volumes of neighborhoods of sub-$\beta$ tangencies] \label{l:det_cov}
Let $x \in \M \setminus \{V=0\}$, $A, B >0$ be small and $u \in B(R)$ be fixed.
Then there exists $C, \sigma_1, \sigma_2>0$ and $T_0(V,x)>0$ such that for all $T \geq T_0$,
\begin{align*}
& \mathbb{P} \left[ \mathfrak{g}_{\alpha} : |\mathfrak{G}(u)| < A, | \det D \mathfrak{G}(u)| < B \} \right] \leq \\
& C A^2 \cdot \sqrt{\mathcal{O}\left(\frac{R}{T}\right) A^{1 - \sigma_1} + \mathcal{O}\left(\frac{R^2}{T^2}\right)A^2 + \mathcal{O}\left(\frac{R^2}{T^2}\right)A^{-2\sigma_1} + A^{1 - \sigma_2} + B },
\end{align*}
with $\mathfrak{G}(\cdot)$ as in \eqref{eq:frakG def} (implicitly depending on $x$ and $T$),
and constants involved in the $``\Oc"$-notation depending on $\alpha$ and $V$ only.
\end{lem}
\begin{proof}
Consider the smooth Gaussian field
\begin{equation*}
H:=\left( \mathfrak{g}_{\alpha}, \nabla \mathfrak{g}_{\alpha}, \partial_{1,2} \mathfrak{g}_{\alpha}, \partial_{1,1} \mathfrak{g}_{\alpha} \right).
\end{equation*}
The distribution of $H(u)$ is non-degenerate Gaussian for every $u\in\R^{2}$ and $\alpha\in [0,1]$ (for $\alpha<1$ it follows from the axiom
$(\rho 4^{*})$, whereas for the most subtle $\alpha=1$ the non-degeneracy of the distribution was shown in \cite[Appendix 1]{CMW16} via an explicit computation; it is no longer non-degenerate if $\partial_{2,2}\mathfrak{g}_{\alpha}$ is added to the vector). The field $H$ proves useful for the following reasons. By rotating the plane $\R^{2} \simeq T_x \M$
if necessary (and using the rotation invariance of the law of $\mathfrak{g}_{\alpha}(\cdot) $), we may assume that
\begin{equation*}
\tilde{V}_{x,T} = (1 + a_1(u)) \partial_1 + a_2(u) \partial_2
\end{equation*}
where $a_1=a_2= \mathcal{O}(R/T)$ and have the additional property of $a_1(0)=a_2(0)=0$. Thus, thanks to the stationarity of $\mathfrak{g}_{\alpha}$ (and hence of $\mathfrak{G}$), we have
\begin{align*}
& \mathbb{P} \left[ \{ |\mathfrak{G}(u)| < A, | \det D \mathfrak{G}(u)| < B \} \right] =
\mathbb{P} \left[ \{ |\mathfrak{G}(0)| < A, | \det D \mathfrak{G}(0)| < B \} \right] \\&=
 \mathbb{P} \left[ \{ |(\mathfrak{g}_{\alpha}(0), \partial_1 \mathfrak{g}_{\alpha}(0))| < A, |\det D \mathfrak{G}(0)| < B \} \right]
\end{align*}
where
\begin{align*}
\det D \mathfrak{G}(0) & = \partial_2 \mathfrak{g}_{\alpha} \cdot \partial_{1,1} \mathfrak{g}_{\alpha} - \partial_{1} \mathfrak{g}_{\alpha} \cdot \partial_{2,1} \mathfrak{g}_{\alpha} + \left( \partial_1 a_1 - \partial_2 a_2 \right) \cdot (\partial_1 \mathfrak{g}_{\alpha}) \cdot (\partial_2 \mathfrak{g}_{\alpha}) \\ & + \partial_1 a_2 \cdot (\partial_2 \mathfrak{g}_{\alpha})^2 - \partial_2 a_1 \cdot (\partial_1 \mathfrak{g}_{\alpha})^2_{|u=0}.
\end{align*}

It follows that,
for some absolute constant $c>0$,
\begin{align} \label{e:ub_vol_bound}
\nonumber & \mathbb{P} \left[ \{  |\mathfrak{G}(u)| < A, | \det D \mathfrak{G}(u)| < B \} \right]
 \leq cA^2 \cdot \mathbb{P} \left[ \{ | \det D \mathfrak{G}(0)| < B \, \bigg| \, \left| \left( \mathfrak{g}_{\alpha}(0), \partial_1 \mathfrak{g}_{\alpha}(0) \right) \right| < A  \} \right] \\
& \leq cA^2 \cdot \mathbb{P} \left[ \{ | \partial_2 \mathfrak{g}_{\alpha} \cdot \partial_{1,1} \mathfrak{g}_{\alpha} - \partial_{1} \mathfrak{g}_{\alpha} \cdot \partial_{2,1} \mathfrak{g}_{\alpha} | < B + F(u) \, \bigg| \, \left| \left( \mathfrak{g}_{\alpha}(0), \partial_1 \mathfrak{g}_{\alpha}(0) \right) \right| < A  \} \right],
\end{align}
where $F(u) = \mathcal{O}\left(\frac{R^2}{T^2}\right) \cdot A \cdot \partial_2 \mathfrak{g}_{\alpha} + \mathcal{O}\left(\frac{R^2}{T^2}\right) \cdot (\partial_2 \mathfrak{g}_{\alpha})^2 + \mathcal{O}\left(\frac{R^2}{T^2}\right) \cdot A^2$ is an estimate on the remaining terms appearing the expression for $\det D \mathfrak{G}(u)$. Note that given $\sigma_1, \sigma_2 > 0$, there exist constants $C_1,C_2,c_1,c_2 >0$, independent of $A$, such that $\mathbb{P} \left[ \bigg| \partial_2 \mathfrak{g}_{\alpha}(0)| > A^{-\sigma_1} \bigg| |\mathfrak{G}(0)| < A \right] \leq C_1 e^{-\frac{c_1}{A^{2\sigma_1}}}$ and $\mathbb{P} \left[ | \partial_{2,1} \mathfrak{g}_{\alpha}(0)| > A^{-\sigma_2} \bigg| |\mathfrak{G}(0)| < A\right] \leq C_2 e^{-\frac{c_2}{A^{2\sigma_2}}}$ thanks to the non-degeneracy of our conditional covariance matrices which follows directly from the non-degeneracy of that of $\tilde{\mathfrak{G}}$.

Let us combine the right-hand side of \eqref{e:ub_vol_bound} with our estimate on $F(u)$, the estimates for the probabilities $\mathbb{P} \left[ \bigg| \partial_2 \mathfrak{g}_{\alpha}(0)| > A^{-\sigma_1} \bigg| |\mathfrak{G}(0)| < A \right]$ and $\mathbb{P} \left[ | \partial_{2,1} \mathfrak{g}_{\alpha}(0)| > A^{-\sigma_2} \bigg| |\mathfrak{G}(0)| < A\right]$, and the explicit quantity of $\mathbb{P} \bigg[ \left| \left( \mathfrak{g}_{\alpha}(0), \partial_1 \mathfrak{g}_{\alpha}(0) \right) \right| < A \bigg]$.  Thus we have reduced to calculating
\begin{align*}
\mathbb{P} \bigg[ | \partial_2 \mathfrak{g}_{\alpha} \cdot \partial_{1,1} \mathfrak{g}_{\alpha} | & \leq \mathcal{O}\left(\frac{R^2}{T^2}\right) A^{1 - \sigma_1} + \mathcal{O}\left(\frac{R^2}{T^2}\right)A^2 + \mathcal{O}\left(\frac{R^2}{T^2}\right)A^{-2\sigma_1} + A^{1 - \sigma_2} + B  \\
& \bigg| \, \left| \left( \mathfrak{g}_{\alpha}(0), \partial_1 \mathfrak{g}_{\alpha}(0) \right) \right| < A \bigg].
\end{align*}
The fact that at least one of the terms $\partial_2 \mathfrak{g}_{\alpha}, \partial_{1,1} \mathfrak{g}_{\alpha}$ must be bounded above by the squareroot of the right-hand side along with the non-degeneracy of the corresponding conditional covariance matrix gives the desired upperbound of
\begin{equation*}
2 \sqrt{\mathcal{O}\left(\frac{R^2}{T^2}\right) A^{1 - \sigma_1} + \mathcal{O}\left(\frac{R^2}{T^2}\right)A^2 + \mathcal{O}\left(\frac{R^2}{T^2}\right)A^{-2\sigma_1} + A^{1 - \sigma_2} + B }.
\end{equation*}
Now, set $\sigma_1=\sigma_2=1/2$ (although this may not be an optimal choice of parameters).
\end{proof}

\begin{proof}[Proof of Proposition \ref{p:abc_stability}]

Our proof follows the proof of \cite[Proposition 4.7]{SW18} with some important modifications of the argument. The main idea is to show that if the count for critical points of $\mathfrak{G}$, that are also low-lying in the sense that the first component of $\mathfrak{G}$ is small, is too large, then for sufficiently small parameters $\beta_{2,2}$, the smoothness of the field $\mathfrak{G}$ allows us to derive contradictions in volume comparisons for neighborhoods of such critical points.

First, we restrict ourselves to the event $\left( \cap_{i=1}^6 \Delta_i^{\complement}\right)$ where $R \geq R_0(\eta,\delta)$, $\beta_{1,3}=\beta_{1,3}(\eta,\delta)$; the parameters governing the events $\Delta_i$, $i=1,...,5$, have been chosen such that $\cap_{i=1}^5 \Delta_i^{\complement}$ has probability $> 1 - \frac{\delta}{8}$. Note that $\mathbb{P}[\Delta_6(\eta, R, \beta_{1,3})^{\complement}] > 1 - \frac{\delta}{8}$. We are now in the situation where the nodal components of $\mathfrak{g}_{\alpha}$ are stable in the sense of Definition \ref{d:unif_stability_fields}.
For a parameter $\beta_{2,2}$, we let $K$ be the number of $\beta_{2,2}$-unstable balls for $\mathfrak{G}$. We will eventually be able to find a number $\beta_{2,2}$ that satisfies the conclusion of Proposition \ref{p:unstable_tangencies}.

Next, we introduce a small parameter $\gamma = \gamma(R_1)$ that we will specify towards the end of our proof. We set
\begin{equation}
\label{eq:Ktild=K}
\tilde{K} = c_1 K
\end{equation}
after invoking the separation statement Lemma \ref{l:4_separation}, and consider the new associated covering $\{ \mathcal{G}_{i_j} \}_{i_j \leq \tilde{K}}$, with centers $u_{i_j}$, where we have possibly reordered the indices $i_j \leq \tilde{K}$. Now Taylor expand $\det D \mathfrak{G}$ to obtain the bound
\begin{equation} \label{e:DG_expansion}
|\det D \mathfrak{G}(u)| < \beta_{2,2} + C_2 \bigg( \sup\limits_{|\nu| = 1, |\tau|=1 ,\, x \in B(u_{i_j}, \gamma)} |\partial^{\nu}  \tilde{V}_{x,T}\mathfrak{g}_{\alpha}(u) | \, |\partial^{\tau}  \tilde{V}_{x,T}\mathfrak{g}_{\alpha}(u)| \bigg) \cdot \gamma
\end{equation}
on each ball $B(u_{i_j}, \gamma)$.
As a result of an application of Lemma \ref{l:derivative_bounds} on \eqref{e:DG_expansion}, we obtain the numbers
\begin{align*}
A &:=  \beta_{2,2} + c_2 \frac{R}{\tilde{\sqrt{K}}}\cdot \gamma \\
B &:= \beta_{2,2} + c_3^2 \frac{R^2}{\tilde{K}} \gamma
\end{align*}
with $c_{2}, c_3$ absolute, so that with probability $1 - \delta/4$ for $R \geq R_1$, the bounds $|\mathfrak{G}(u)| < A$ and $|\det D \mathfrak{G}(u)| < B$ hold on at least half of the $B(u_j, \gamma)$, assuming $R \geq R_1$
(though we may neglect the distinction between ``all of the balls" and ``half of them" by appropriately modifying
the constant in \eqref{eq:Ktild=K}).

Our third, and most important, step is to manipulate the random variable
\begin{equation*}
\mathcal{A}_{A,B}  = \mbox{Area} \left( \{ u:  |\mathfrak{G}(u)|<A, |\det D\mathfrak{G}(u)| < B \} \right)
\end{equation*}
in order to obtain a sufficient estimate on $K$ and deduce some necessary restrictions on $\beta_{2,2}$.
We now have with probability $> 1 - \delta/4$, following immediately from the assumptions on the balls $B(u_j, \gamma)$ and that $\tilde{K} = c_1K$, that there exists $c_3 > 0$ uniform such that the following bound holds:
\begin{equation}
\label{eq:Ac(A,B)>=cgamma^2}
\mathcal{A}_{A,B} \geq c_3 \gamma^2 \cdot K.
\end{equation}

We have
\begin{equation*}
\mathcal{A}_{A,B} = \int\limits_{B(R)} \chi_{A,B} \left( g_1,g_2, \frac{\partial g_1}{\partial x_1}  \frac{\partial g_2}{\partial x_2} -  \frac{\partial g_1}{\partial x_2} \frac{\partial g_2}{\partial x_1} \right) \, dx
\end{equation*}
where $\chi_{A,B}(\cdot, \cdot, \cdot)$ is the indicator function for the set $\{|(v_1, v_2)| < A, \, |v_3| < B \}$. If we can obtain a sufficient bound on $\E[\mathcal{A}_{A,B}]$ then, after an application of Chebyshev's Inequality, we have a sufficient estimate on $\mathcal{A}_{A,B}$ with high probability. This is where Lemma \ref{l:det_cov} is invoked, since
\begin{equation*}
\E \left[\chi_{A,B} \left( g_1,g_2, \frac{\partial g_1}{\partial x_1}  \frac{\partial g_2}{\partial x_2} -  \frac{\partial g_1}{\partial x_2} \frac{\partial g_2}{\partial x_1} \right) \right] =  \mathbb{P} \left[ \mathfrak{g}_{\alpha} : |\mathfrak{G}(u)| < A, | \det D \mathfrak{G}(u)| < B \} \right].
\end{equation*}

We remind ourselves (as already done in the proof of Theorem \ref{t:conv_prob}) that in the local coordinates around $x$ given by $y=\exp_x(Y)$, which allows us to identify $\R^2$ with $T_xM$, we set $\tilde{V}(x)$ to be the constant vector field given by the trivial extension of $(\tilde{V}_{x,T})_{|Y=0}$ to $\R^2$.  We can choose $T \geq T_1(V(x))$ such that $\tilde{V}_{x,T} = \tilde{V}(x) + o(1)$; we remind ourselves that we have also assumed $\tilde{V}(x) \neq 0$. Note our $T_1$ can be taken to be greater than the $T_0$ appearing in the statement of Lemma \ref{l:det_cov}. Let $s_{1},s_{2}$ be any positive numbers satisfying $2s_{1}+s_{2}=2$, so that $\gamma^{-2} = \gamma^{-2s_1 - s_2}$.
The estimate given by Lemma \ref{l:det_cov} together with \eqref{eq:Ac(A,B)>=cgamma^2}, yields that
\begin{align} \label{e:penult_step_small_number_K}
\nonumber K & \leq C_3 \gamma^{-2}  \,  \left(  \beta_{2,2} + c_2 \frac{R}{\sqrt{\tilde{K}}} \gamma  \right)^2 \\
\nonumber & \times \sqrt{\mathcal{O}\left(\frac{R}{T}\right) A^{1/2} + \mathcal{O}\left(\frac{R^2}{T^2}\right)A^2 + \mathcal{O}\left(\frac{R^2}{T^2}\right)A^{-1} + A^{1/2} +  \left( \beta_{2,2} + c_3^2 \frac{R^2}{\tilde{K}} \gamma \right)  } \times R^2 \\
\nonumber & = C_3 \,  \left(\beta_{2,2}\gamma^{-s_1} + c_2 \frac{R}{\sqrt{\tilde{K}}} \gamma^{1-s_1}   \right)^2 \cdot \bigg( \mathcal{O}\left(\frac{R}{T}\right) A^{1/2} \gamma^{-2s_2} + \mathcal{O}\left(\frac{R^2}{T^2}\right)A^2\gamma^{-2s_2} + \mathcal{O}\left(\frac{R^2}{T^2}\right)A^{-1} \gamma^{-2s_2} \\
& + A^{1/2}  \gamma^{-2s_2}+  \left( \beta_{2,2}\gamma^{-2s_2} + c_3^2 \frac{R^2}{\tilde{K}} \gamma^{1-2s_2} \right)  \bigg)^{\frac{1}{2}} \times R^2.
\end{align}
In preparation for the final steps of our proof, let us define
\begin{align} \label{e:xi_value}
\nonumber & \xi := C_3 \left(\beta_{2,2}\gamma^{-s_1} + c_2 \frac{R}{\sqrt{\tilde{K}}} \gamma^{1-s_1}   \right)^2 \cdot \bigg( \mathcal{O}\left(\frac{R}{T}\right) A^{1/2} \gamma^{-2s_2} + \mathcal{O}\left(\frac{R^2}{T^2}\right)A^2\gamma^{-2s_2} + \mathcal{O}\left(\frac{R^2}{T^2}\right)A^{-1} \gamma^{-2s_2} \\
& + A^{1/2}  \gamma^{-2s_2}+  \left( \beta_{2,2}\gamma^{-2s_2} + c_3^2 \frac{R^2}{\tilde{K}} \gamma^{1-2s_2} \right)  \bigg)^{\frac{1}{2}}
\end{align}
so that we have $K \leq \xi R^2$; without loss of generality, if necessary, we increase $C_3$ to be greater than $1$.

Let us now proceed by contradiction. To wit, let us assume that $K > \eta R^2$ from which it follows that $\frac{R}{\sqrt{\tilde{K}}} < \frac{1}{c_1 \sqrt{\eta}}$, on recalling \eqref{eq:Ktild=K}. In order to make our choice of parameters $\gamma, \beta_{2,2},$ and $T_0$ properly so that we can eventually make $\xi < \eta$, we proceed as follows.
\begin{enumerate}
\item To start, we let $\gamma <1$ and $\beta_{2,2} < 1$.
We would like to make $A^{1/2} \gamma^{-2s_2} < \frac{1}{100}$. Therefore when given $\eta$, this leads us the first restriction of
\begin{equation*}
\gamma^{1-4s_2} < \frac{1}{100} \frac{c_1 \sqrt{\eta}}{c_2},
\end{equation*}
therefore requiring that $1 - 4s_2 > 0 $.
\item We aim to have that $A^2 \gamma^{-2s_1} =  \left(\beta_{2,2}\gamma^{-s_1} + c_2 \frac{R}{\sqrt{\tilde{K}}} \gamma^{1-s_1}   \right)^2 $ in the r.h.s of (\ref{e:xi_value}) is less than $\frac{\eta}{100C_3}$. This requires that
\begin{equation*}
\gamma^{1-s_1} < \frac{c_1 \,\eta}{100 \, c_2 \, C_3}
\end{equation*}
therefore requiring that $1 - s_1 > 0 $. As we have the initial restriction of $2s_1 + s_2 =2$, as a concrete working set of parameters, we make the choice of $s_1 = 15/16$ and $s_2 = 1/8$. Hence, we take
\begin{equation*}
\gamma  < \min \left\{ \left(\frac{c_1 \, \sqrt{\eta}}{100 \, c_2 \, C_3} \right)^{2}, \left(\frac{c_1 \eta}{100 \, c_2 \, C_3} \right)^{16} \right\} = \left( \frac{c_1}{100 c_2 C_3} \right)^{16} \eta^{16}.
\end{equation*}
\item With the goals of eventually having $A^2 \gamma^{-2s_1} < \frac{\eta}{100C_3}$ in the r.h.s of (\ref{e:xi_value}) and the second (non-constant) factor that is raised to the power of $1/2$ be $< 1$, we make another restriction of
\begin{align*}
& \beta_{2,2} < \min \left\{ \frac{\eta^{1/2}}{\sqrt{100C_3}} \gamma^{s_1}, \frac{1}{10} \gamma^{2s_2} \right\} =  \min \left\{ \frac{\eta^{1/2}}{100C_3} \gamma^{15/16}, \frac{1}{100} \gamma^{1/4} \right\} \\
& = \frac{\eta^{1/2}}{100C_3} \gamma^{15/16} < \frac{1}{100C_3} \,  \left( \frac{c_1}{100 c_2 C_3} \right)^{15} \eta^{31/2}.
\end{align*}
after using the restriction on $\gamma$.  Note that our restrictions on $\gamma$ and $\beta$ make all terms which do not depend on $T$ sufficiently small.
\item Now, considering all the terms in (\ref{e:xi_value}) involving factors of $T^{-1}$, appropriately bounding the term $\mathcal{O}\left(\frac{R^2}{T^2}\right)A^{-1} \gamma^{-2s_2}$ above by $\frac{1}{100}$ will place similar bounds on the aforementioned remaining terms.  Letting $C$ be the supremum amongst all the implicit constants depending on a derivative of $V$ (that is, the $C^1$ norm of $V$'s coefficients) and appearing in the $\mathcal{O}$ notations, we make the final restriction of
\begin{equation*}
\frac{C}{T} < \frac{\sqrt{A}}{10R} \gamma^{2s_2}= \frac{\sqrt{A}}{10R} \gamma^{1/4} < \frac{1}{10R} \left(\frac{c_1}{c_2 C_3} \right)^4 \eta^{15/4}.
\end{equation*}
\end{enumerate}
Hence, we enlarge $T_0$ to satisfy our final constraint (and therefore inheriting a possible dependence on $R$). With all these choices made, in the order described above, we have arrive at $\xi < \frac{\eta}{2500}$ and therefore contradict the hypothesis of $K > \eta R^2$. This concludes the proof of Proposition \ref{p:abc_stability}.
\end{proof}

We now collect some facts on nodal components that fall outside of our local approximation and are exceptional in terms of the volumes of the nodal domains for which they form the boundary. But first, we record some necessary definitions:

\begin{defn}
\begin{enumerate}
\item We say that a nodal component of $\mathfrak{g}_{\alpha}^{-1}(0)$ is $\xi$-small if it is adjacent to a domain of volume $< \xi$.

\item For $R >0$, let $\nod_{\xi-sm}(\mathfrak{g}_{\alpha}, R)$ be the number of $\xi$-small components of $\mathfrak{g}_{\alpha}^{-1}(0)$ lying entirely inside of $B(R)$.

\item We say that a nodal component of $\mathfrak{g}_{\alpha}^{-1}(0)$ is $R_1$-long if its diameter is $> R_1$.

\item For $R > R_1 >0$, let $\nod_{diam > R_1}(\mathfrak{g}_{\alpha};R)$ be the number of $R_1$-long components of $\mathfrak{g}_{\alpha}^{-1}(0)$ lying entirely inside of $B(R)$.
\end{enumerate}
\end{defn}

The following results are taken directly from \cite{So12}:

\begin{lem}[Cf. {\cite[Lemma 8]{So12}}] \label{l:long_components}
Consider $\mathfrak{g}_{\alpha}$ on $B(R)$. Let $\delta>0$ be given. Then given $R_1 >0$, there exists $R_0(R_1)$ such that for all $R \geq R_0 \geq R_1$ and a uniform constant $C>0$, such that the number of components $\Gamma \subset B(R)$ of $\mathfrak{g}_{\alpha}^{-1}(0)$ of diameter $> R_1$ is
\begin{equation*}
\nod_{diam > R_1}(\mathfrak{g}_{\alpha};R) < \frac{C}{R_1} R^2
\end{equation*}
with probability $> 1 - \delta$.
\end{lem}

In our $2$-dimensional case Lemma \ref{l:long_components} follows from the isoperimetric inequality; in the more general setting a slightly heavier
machinery involving the restriction of $\mathfrak{g}_{\alpha}$ to the boundary of $B(R)$ is employed. We note that it, thanks to discarding the long components, coverings of $B(R)$ with balls of radius $R_1 < R$ as in the previous sequence of statements will give us an advantage.

\begin{lem}[Cf. {\cite[Lemma 9]{So12}}] \label{l:small_vol}
Let $\xi>0$ be given.  There exist constants $c_0, C_0>0$ such that
\begin{equation*}
\limsup\limits_{R \rightarrow \infty} \frac{\E \left[ \nod_{\xi-sm}(\mathfrak{g}_{ \alpha}, R) \right]}{R^2} \leq C_0 \, \xi^{c_0}.
\end{equation*}
\end{lem}

Given our facts about $\xi$-small and $R_1$-long components, we can now state the following important lemma:

\begin{lem} \label{l:sub_tang_low_count}
Let $x \in \M\setminus\{V=0\}$. Given the parameters $R, \delta, \eta_0>0$, there exists a covering radius parameter $R_{1,0} < R$, an auxiliary density parameter $\eta_1(R, \eta_0) < \eta_0$, an event $E$ with $\mathbb{P}[E] < \delta$ and a stability parameter $\beta_{2,2}$ such that outside $E$, if $F$ is $(\eta_1, \beta_{2,2}, 3R_{1,0})$-stable for all $T \geq T_0(R)$, then $|\det D \mathfrak{G}| > \beta_{2,2}$ on all but $\eta_0 R^2$ components of $\mathfrak{g}_{\alpha}^{-1}(0)$ that are also uniformly stable as in Definition \ref{d:unif_stability_fields}.
\end{lem}

\begin{proof}
Our proof is almost verbatim that of \cite[Lemma 4.8]{SW18} which is similar in spirit to that in \cite[Section 4.2]{NS09}. We refer the interested reader to this paper but highlight the main details as well as an explicit formula which gives further insight into the dependences between our variety of parameters. To ease the notation and intuitively connect the parameters in this lemma with those in our previous statements, we replace $R_{1,0}$ by $R_1$ in our proof.

Thanks to lemmas \ref{l:long_components} and \ref{l:small_vol}, we know that with probability $1  - \frac{\delta}{2}$ there exists uniform constants $C_1,C_2, c_0>0$ such that we have that
\begin{equation} \label{e:long_toosmall_comps}
\nod_{diam>R_1}(\mathfrak{g}_{\alpha}, R) \leq C_1 \, \frac{R^2}{R_1} \, \mbox{ and } \, \nod_{\xi_0-sm}(\mathfrak{g}_{\alpha}, R) \leq C_2 \, \xi^{c_0} \, R^2
\end{equation}
for some $R_1 < R$, where $R \geq R_0$. This reduces us to now having to count components that are not $R_1$-long or are not $\xi_0$-small.  This particular $R_1$ is what we will take as our covering radius parameter, for at least the time being.

Proposition \ref{p:unstable_comp} states that with probability $1 - \frac{\delta}{2}$ and $R_0$ possibly larger, the total number of components that are not long or small in the above senses and are stable in terms of having low-lying zeroes of $\nabla \mathfrak{g}_{\alpha}$ is $\leq \eta_1 R^2$.  Now, we apply Proposition \ref{p:abc_stability} and the uniform lower bound we have on the volume encompassed by the individual components to see that the number of components that are neither long nor small with no sub-$\beta_{2,2}$ $\tilde{V}_{x,T}$- tangencies is $\leq C_3 \frac{R_1^2}{\xi_0} \frac{\eta_1}{2} R^2$ for some uniform $C_3 >0$. Summing up all these counts gives an upper bound for the number of components that are either unstable in the sense of Definitions \ref{d:unif_stability_fields} and \ref{d:unstable_event} and those that are not amenable to our local methods.

By summing all of our declared estimates appearing in (\ref{e:long_toosmall_comps}) and the previous paragraph, and therefore collating all the parameters we can adjust, we aim to establish the following inequality:
\begin{equation*}
C \left( \frac{1}{R_1} + \xi_0^{c_1} + \frac{R_1^2}{2 \xi_0}\eta_{1}  + \eta_1 \right) < \eta_0
\end{equation*}
for some uniform constants $C, c_1>0$. Motivated by this goal, we can choose the range for our medley of parameters in the following order: our radius parameter $R_1 > \frac{4C}{ \eta_0}$, $\xi_0 < (\frac{\eta_0}{4C})^{1/c_1}$, and finally our auxiliary density parameter $\eta_1 < \min \{\frac{\eta_0}{4C}, \frac{2 \xi_0 }{R_1^2} \frac{\eta_0}{4C} \}$. Our choice of $T_0$ comes from Proposition \ref{p:abc_stability}. This gives our desired result.
\end{proof}

\begin{rem}
Note that in the above proof, our choice of parameters $R_1 = R_{1,0}, \xi_0, \eta_1$ is not dictated by our choice of stability parameter $\beta_{2,2}$ and hence not making our reductions circular. However, these choices of parameters could possibly make $\beta_{2.2}$ smaller. Given the restrictions we derived in the proof of Proposition \ref{p:abc_stability}, this does not harm our overall argument.
\end{rem}

Finally, we can deduce Proposition \ref{p:unstable_tangencies} from Lemma \ref{l:sub_tang_low_count}.

\begin{proof}[Proof of Proposition \ref{p:unstable_tangencies}]
We apply Lemma \ref{l:sub_tang_low_count} directly. Considering that we are given $\eta, \delta>0$ and $x$, we take any $R_0 > R_{1,0}$ such that $\frac{R_0}{R_{1,0}}$ is large, take $\eta_0 < \eta$, set $\beta_2 = \beta_{2,2}$, and take the same $T_0$ as provided.
\end{proof}

We close this section by emphasizing that our count of $\eta R^2$ is for the total number of components with at least one near-degenerate tangency.  Thus, there is still room to make more precise this upper bound for those with precisely $k$ tangencies with at least one being near-degenerate.

\section{Proof of Theorem \ref{thm:main meas conv}: global results in the Riemannnian scenario}

\subsection{Proof of Theorem \ref{thm:main meas conv} assuming Theorem \ref{t:main_thm_less_equal_1}: consolidating the individual counts:}

\label{sec:proof main thm meas}

\begin{proof}

Using the constants $C_{\alpha,k}$ prescribed by Theorem \ref{t:main_thm_less_equal_1},
we define the measure
\begin{equation}
\label{eq:mualpha def}
\mu_{\alpha}:=\sum\limits_{k=0}^{\infty}C_{\alpha,k}\cdot\delta_{k},
\end{equation}
and Theorem \ref{t:main_thm_less_equal_1} part (1) reads for every $k\ge 0$,
\begin{equation*}
\E[|\left(\mu_{f_{\alpha,T}}(V)\right)(k) - \mu_{\alpha}(k) |] \rightarrow 0
\end{equation*}
as $T\rightarrow\infty$.
We claim that Theorem \ref{thm:main meas conv} holds true with $\mu_{\alpha}$ as in \eqref{eq:mualpha def}.
First, since by the proof of Theorem \ref{t:main_thm_less_equal_1} below, the constants $C_{\alpha}$ are
same as in Theorem \ref{thm:euclid_count} (1)
applied on the random field $F=\gfr_{\alpha}$ with $\zeta$ arbitrary (cf. Proposition \ref{p:almost_main_theorem}), where
the $C_{k}$ are independent of $\zeta$ by Theorem \ref{thm:euclid_count} (2), the measure
$\mu_{\alpha}$, as defined in \eqref{eq:mualpha def}, is a probability measure, thanks to Theorem \ref{thm:euclid_count} (4).

Therefore $\{\mu_{f_{\alpha,T}}(V)\}_{T}$ on $\Z_{\ge 0}$ is a collection of random probability measures,
and $\mu_{\alpha}$ a deterministic probability measure, so that,
for every $k\ge 0$, as $T\rightarrow\infty$, $(\mu_{f_{\alpha,T}}(V))(k)$ converges in mean (and hence in probability) to $\mu_{\alpha}(k)$.
This is precisely the scenario considered by ~\cite[Lemma 6.3]{ALL}, proof (and same result)
already contained within ~\cite[Proof of Theorem 1.1, \S 7.1 on p. 57]{SW18},
whose conclusion yields the convergence in probability of
$\{\mu_{f_{\alpha,T}}(V)\}_{T}$ as $T\rightarrow\infty$ to $\mu_{\alpha}$ w.r.t. the total variation distance \eqref{eq:D tot var dist def},
i.e. it gives \eqref{eq:main meas conv}.
Finally, the support statement of Theorem \ref{thm:main meas conv} for $\mu_{\alpha}$ is an immediate corollary from
the definition \eqref{eq:mualpha def} of $\mu_{\alpha}$ and Theorem \ref{t:main_thm_less_equal_1} part (2).
Theorem \ref{thm:main meas conv} is now proved.

\end{proof}

\subsection{Proof of Theorem \ref{t:main_thm_less_equal_1}}

\begin{defn}\label{def:plus_minus}
For $y \in \R$, we denote
\begin{equation*}
|y|_+ = \max \{0,y\}
\end{equation*}
and
\begin{equation*}
|y|_-= \min \{0,y\},
\end{equation*}
so that
\begin{equation*}
| \cdot | = | \cdot |_+ + | \cdot |_-.
\end{equation*}
\end{defn}

Let $V \in \mathcal{V}(\M)$ throughout this section. We remind ourselves that
\begin{equation*}
\nod_{V}(f,k) :=\#\{\gamma  \subseteq f^{-1}(0):\: \nod_{V}(\gamma)=k\}
\end{equation*}
is the total number of nodal components of $f$ with precisely $k$ tangencies w.r.t. $V$. We now present a key proposition that builds upon on our previous (and crucial) local results and essentially gives the conclusions of our main theorems:

\begin{prop} \label{p:almost_main_theorem}
Let $k \ge 0$ be given and $c_{2, \alpha}$ be the Nazarov-Sodin constant of $\mathfrak{g}_{\alpha}$,
and the constants $C_{\alpha,k}$ prescribed by an application of Theorem \ref{t:main_thm_less_equal_1} (1) on the
field $\gfr_{\alpha}$ (with $\zeta$ arbitrary). Then, as $T \rightarrow \infty$, we have that
\begin{equation*}
\E \left[  \left| \frac{ \nod_{V}\left( f_T,  k  \right)}{c_{2, \alpha} \cdot \vol(\M) \cdot T^2} - C_{\alpha,k} \right|_{\pm}  \right] \rightarrow 0.
\end{equation*}
\end{prop}

\begin{proof}[Proof of Theorem \ref{t:main_thm_less_equal_1} assuming Proposition \ref{p:almost_main_theorem}]

Theorem \ref{t:main_thm_less_equal_1} follows immediately from Proposition \ref{p:almost_main_theorem}.

\end{proof}

\subsection{Excising small volume and long components}

\begin{defn}
Let $\xi, D> 0$ be parameters and $f_T$ the band-limited functions \eqref{eq:fT band lim alpha<1} (or \eqref{eq:fT band lim alpha=1}).

\begin{enumerate}
\item A component of $f^{-1}_T(0)$ is \textit{$\xi$-small} if it is a boundary of nodal domain whose volume in $\M$ is less than $\xi T^{-2}$.  Let $\nod_{\xi-sm}(f_T)$ be the total number of $\xi$-small components of $f_T$ on $\M$.
\item For $D>0$, a component of $f^{-1}_T(0)$ is \textit{$D$-long} if its diameter is greater than $D/T$. Let $\nod_{D-long}(f_T)$ be their total number.
\item Given the parameters $D, \xi>0$, a component of $f^{-1}_T(0)$ is \textit{$(D, \xi)$-normal}, if it is neither $\xi$-small nor $D$-long.
\item Let $\nod_{norm}(f_T)$ be the total number of $(D,\xi)$-normal components of $f^{-1}_T(0)$ (analogously, define $\nod_{V, norm}(f_T, k)$ as the count for the subset of $(D,\xi)$-normal components whose number of tangencies is precisely $k$ with respect to $V$).
\item  Let $\nod_{V, norm}(f_T, x, r)$ be the total number of $(D,\xi)$-normal components of $f^{-1}_T(0)$ that are completely contained inside the geodesic ball $B(x;r)$ (respectively, $\nod_{V,norm}(f_T, k, x, r)).$
\item  Let $\nod^*_{norm}(f_T, x, r)$ be the total number of $(D,\xi)$-normal components of $f^{-1}_T(0)$ that intersect the geodesic ball $B(x;r)$ (respectively, $\nod^*_{V,norm}(f_T, k, x, r)).$
\end{enumerate}
\end{defn}

The following two lemmas are taken directly out of \cite{So12}:

\begin{lem}[Cf. {\cite[Lemma 9]{So12}}]
There exists constants $c_0, C_0 >0$ so that the following estimate on the number of $\xi$-small components holds:
\begin{equation*}
\limsup\limits_{T \rightarrow \infty} \frac{\E \left[  \nod_{\xi-sm}(f_T) \right]}{T^2} \leq C_0 \, \xi^{c_0}.
\end{equation*}
\end{lem}

\begin{lem}[Cf. {\cite[Lemma 8]{So12}}]
There exists a constant $C_0 >0$ such that the following bound holds on the number of $D$-long components:
\begin{equation*}
\limsup\limits_{T \rightarrow \infty} \frac{\E \left[  \nod_{D-long}(f_T) \right]}{T^2} \leq C_0 \, \frac{1}{D}.
\end{equation*}
\end{lem}

We can now state the main proposition of this section:

\begin{prop} \label{p:conv_in_mean_normal}
Let $\xi, D>0$ be given and $k \ge 0$. Then as $T \rightarrow \infty$, we have that
\begin{equation*}
\E \left[  \left| \frac{ \nod_{V,norm}\left( f_T,  k  \right)}{c_{2, \alpha} \cdot \vol(\M) \cdot T^2} - C_k \right|_{\pm}  \right] \rightarrow 0,
\end{equation*}
for $| \cdot |_{\pm}$ as in Definition \ref{def:plus_minus}, and $C_{k}=C_{\alpha,k}$ is in Proposition \ref{p:almost_main_theorem}.
\end{prop}

\begin{proof}[Proof of Proposition \ref{p:almost_main_theorem} assuming Proposition \ref{p:conv_in_mean_normal}]
This is immediate given lemmas \ref{l:small_vol} and \ref{l:long_components}.
\end{proof}

\subsection{Proof of Proposition \ref{p:conv_in_mean_normal}}

We begin with a global analogue of Lemma \ref{l:integro_geom_sandwich}.

\begin{lem} \label{l:integro_geom_sandwich_global}
Given $\epsilon > 0$ and $k \ge 0$, there exists $r_0$ such that for all $r < r_0$, we have
\begin{align*}
& (1 - \epsilon) \int\limits_{\M} \frac{\nod_{V,norm}\left( f_T,  k, x, r  \right)}{\vol(B(r))} \, dx \leq \nod_{V,norm}\left( f_T,  k \right)\leq (1 + \epsilon) \int\limits_{\M} \frac{\nod^*_{V,norm}\left( f_T,  k, x, r  \right)}{\vol(B(r))} \, dx.
\end{align*}
\end{lem}

\begin{proof}
As the proof is almost exactly the same as that of Lemma \ref{l:integro_geom_sandwich}, we omit it and leave the details to the reader.
\end{proof}

\begin{proof}[Proof of Proposition \ref{p:conv_in_mean_normal}]
Once again, our proof mirrors that of \cite[Proposition 3]{BW17} so we give the major steps of the proof, with particular details given at the critical junctures, and leave the remaining similar details to the interested reader.  For convenience, we assume that $\vol(\M)=1$.
Let $\epsilon > 0$ and $k \ge 0$ be given. Consider a regime of $R,T$ such that $R/T < r_0 < 1$ as in Lemma \ref{l:integro_geom_sandwich_global} and set $D = \sqrt{R}$ so that
\begin{equation*}
\frac{\vol(B(R+D))}{\vol(B(R))} < 1 + \epsilon.
\end{equation*}

We focus on the case for $| \cdot |_+$ as the proof for $| \cdot |_{-}$ is similar. We begin by applying Lemma \ref{l:integro_geom_sandwich_global} with $r = R/T$. It follows that
\begin{align*}
 \E \left[  \left| \frac{ \nod_{V,norm}\left( f_T,  k  \right)}{c_{2, \alpha} \cdot \vol(\M) \cdot T^2} - C_k \right|_{+}  \right] &\leq \E \left[  \int\limits_{\M} \left|(1 + 2\epsilon) \frac{\nod^*_{V,norm}\left( f_T,  k, x, R/T  \right)}{c_{2, \alpha} \cdot \vol(B(R+D))} \, dx - C_k  \right|_{+} \, dx  \right] \\
& \leq \E \left[  \int\limits_{\M} \left|\frac{\nod_{V,norm}\left( f_T,  k, x, (R+D)/T  \right)}{c_{2, \alpha} \cdot \vol(B(R+D))} - C_k  \right|_{+} \, dx  \right] \\
& + \mathcal{O} \left( \epsilon \cdot \int\limits_{\M} \frac{\E \left[\nod_{V,norm}\left( f_T,  k, x, (R+D)/T  \right) \right]}{ c_{2, \alpha} \cdot \vol(B(R+D))} \, dx \right).
\end{align*}
Note that
\begin{equation*}
\int\limits_{\M} \frac{\E \left[\nod_{V,norm}\left( f_T,  k, x, (R+D)/T  \right) \right]}{ c_{2, \alpha} \cdot \vol(B(R+D))} \, dx  =  \mathcal{O}(1)
\end{equation*}
by \cite[Lemma 2]{So12} on using Lemma \ref{lem:Kac Rice manifold} to obtain
an upper bound for the total number of components. Thus, our focus is on proving that
\begin{equation*}
\E \left[  \int\limits_{\M} \left|\frac{\nod_{V,norm}\left( f_T,  k, x, (R+D)/T  \right)}{c_{2, \alpha} \cdot \vol(B(R+D))} - C_k  \right|_{+} \, dx  \right]  \rightarrow 0.
\end{equation*}

Let $\rho < \frac{\epsilon}{5\sqrt{\pi} N}$, where $N$ is the cardinality of $\{ V = 0 \}$, be a parameter that will control how close we are allowed to approach the zero set $\{V=0\}$. Let $\xi >0$ be a small number, to be specified shortly, that controls how small a volume a nodal component can encompass and will also be related to our given $\epsilon$ later.
We now consider the event
\begin{equation*}
\Delta_{T,k,x,R} := \left\{ \left| \frac{\nod_{V,norm}\left( f_T,  k, x, (R+D)/T  \right)}{c_{2, \alpha} \cdot \vol(B(R+D))}  - C_k  \right| > \frac{\epsilon}{5} \right\}.
\end{equation*}
Let $\{ V= 0\}^{\sqrt{\rho/2}}$ be a neighborhood consisting of balls of radius $\sqrt{\rho/2} < inj(\M)$ centered at the finitely many points of $\{V=0\}$. An application of an Egorov-type Theorem w.r.t. the double limit $\lim\limits_{R\rightarrow\infty}\lim\limits_{T\rightarrow\infty}$
(as carried out in \cite[Proposition 3]{BW17}), along with Theorem \ref{t:conv_prob} allows us to conclude that on $\M \setminus (\{ V= 0\}^{\sqrt{\rho/2}})$, there exists a set $\M_{\epsilon} \subset \M$ with $\vol(\M_{\epsilon}) > 1 - \sqrt{\pi} \frac{\epsilon}{5}$ (with the additional property that $\M_{\epsilon} \cap \{ V= 0\}^{\sqrt{\rho/2}} = \emptyset$ thanks to taking a tubular neighborhood of radius $\sqrt{\rho/2}$), such that
\begin{equation}
\label{eq:double limit liminf=0}
\liminf_{R \rightarrow \infty} \limsup_{T \rightarrow \infty} \sup\limits_{x \in \M_{\epsilon}} \mathbb{P}\left[  \Delta_{T,k,x,R} \right] = 0,
\end{equation}
i.e. the limit \eqref{eq:double lim R,T conv loc} is {\em almost} uniform along some sequence $R_{j}\rightarrow\infty$ attaining the $\liminf\limits_{R\rightarrow\infty}$ in \eqref{eq:double limit liminf=0}, so that
\begin{equation}
\label{eq:double limit R attain liminf}
\lim\limits_{j \rightarrow \infty} \limsup_{T \rightarrow \infty} \sup\limits_{x \in \M_{\epsilon}} \mathbb{P}\left[  \Delta_{T,k,x,R_{j}} \right] = 0.
\end{equation}
Without loss of generality, we can restrict ourselves to a regime of $R,T$ where there exists a large enough $C=C_{\M} >0$ such that for all ratios $R/T \leq \frac{inj(\M)}{C}$, we have that $\mbox{vol}_g B(x, R/T) = \pi \frac{R^2}{T^2} + o(\frac{R^2}{T^2})$.

We re-express
\begin{align} \label{e:expect_int_geom}
& \E \left[  \int\limits_{\M} \left|\frac{\nod_{V,norm}\left( f_T,  k, x, (R+D)/T  \right)}{c_{2, \alpha} \cdot \vol(B(R+D))} - C_k  \right|_{+} \, dx  \right] \\
\nonumber & = \int\limits_{\Omega} \int\limits_{\M}  \left|\frac{\nod_{V,norm}\left( f_T,  k, x, (R+D)/T  \right)}{c_{2, \alpha} \cdot \vol(B(R+D))} - C_k  \right|_{+} \, dx \, d \mathbb{P}(\omega).
\end{align}
An application of Fubini's Theorem along with the measure bounds for the sets $\Delta_{T,k,x,R}$ and $\M_{\epsilon}$ reduces bounding (\ref{e:expect_int_geom}) to understanding the sum
\begin{equation*}
\left( \int\limits_{\M} \int\limits_{\Delta_{T,k,x,R}} + \int\limits_{\M} \int\limits_{\Omega \cap \Delta_{T,k,x,R}^{\complement}} \right) \underbrace{\left|\frac{\nod_{V,norm}\left( f_T,  k, x, (R+D)/T  \right)}{c_{2, \alpha} \cdot \vol(B(R+D))} - C_k  \right|_{+}\, d \mathbb{P}(\omega)   \, dx }_{=: F \, d \mu} .
\end{equation*}
This the sum of integrals breaks up into a further $5$ integrals in total and we estimate each accordingly:
\begin{equation*}
\int\limits_{\M_{\epsilon}} \int\limits_{\Delta_{T,k,x,R}} F  \, d \mu \leq \xi^{-c_0} \cdot \sup\limits_{x \in \M_{\epsilon}} \mathbb{P}[\Delta_{T,k,x,R}],
\end{equation*}
\begin{equation*}
\int\limits_{\M\setminus\M_{\epsilon}} \int\limits_{\Delta_{T,k,x,R}} F  \, d \mu \leq \frac{\epsilon}{5} \cdot \xi^{-c_0} \cdot \sup\limits_{x \in \M_{\epsilon}} \mathbb{P}[\Delta_{T,k,x,R}],
\end{equation*}
\begin{equation*}
\int\limits_{\M_{\rho}} \int\limits_{\Delta_{T,k,x,R}^{\complement}} F  \, d \mu \leq \frac{\epsilon}{5},
\end{equation*}
\begin{equation*}
\int\limits_{  \{V=0\}^{\sqrt{\rho/2}} } \int\limits_{\Delta_{T,k,x,R}^{\complement}} F  \, d \mu \leq C_{\M} \cdot \frac{\epsilon^2}{50}, \mbox{ thanks to the finiteness of } \{V=0\} \mbox {, and }
\end{equation*}
\begin{equation*}
\int\limits_{\M\setminus \left( \M_{\rho} \cup (\{V=0\}^{\sqrt{\rho/2}} \right)} \int\limits_{\Delta_{T,k,x,R}^{\complement}} F  \, d \mu \leq C_{\M} \cdot \frac{\epsilon^2}{25}, \mbox{ as }  \left( \M\setminus  \M_{\rho} \right) \cap (\{V=0\}^{\sqrt{\rho/2}}  \neq \emptyset \mbox{ is possible}
\end{equation*}
where $C_{\M} > 0$ is an absolute constant depending only on the volume of $\M$, which we have normalized to $1$.

Keeping in mind that $\epsilon$ was given to us, we restrict the auxiliary parameters introduced in our proof as follows: 1) $\xi < (\frac{\epsilon}{5})^{1/c_0}$, 2) take $R_0$ such that for all $R_{j}  \geq R_0$ with $\{R_{j}\}_j$ as in \eqref{eq:double limit R attain liminf} (which in turn, as in the proof of Theorem \ref{t:conv_prob}, requires us to carefully select our parameters $\beta_{1,1}, \beta_{1,2}, \beta_{1,3}, \beta_2$ each of which themselves depends on $\eta$), we know there exists $T_0(R_{j})$ (implicitly depending upon our string of stability parameters $\beta$) such that for all $T \geq T_0$ we have $ \sup\limits_{x \in \M_{\rho}} \mathbb{P}[\Delta_{T,k,x,R_{j}}] < \frac{\epsilon \xi^{c_0}}{5}$. We conclude that \eqref{e:expect_int_geom} is bounded above by $\epsilon$. This completes the proof of our proposition.
\end{proof}

We conclude by noting that our spectral parameter $T_0$ being dependent on the previous parameters $\beta_{1,2}$ and $\epsilon$ is a reflection of how our covering of $\M \setminus \M_{\epsilon}$, by geodesic disks of radius $R/T$, is governed by the maximal order of vanishing amongst all zeroes of $V$.

\end{document}